\documentclass[final]{siamltex}

\newcommand{\p}{{\mathtt{n}_{\mathtt{p}}}}

\usepackage{amsmath} 
\usepackage{amsfonts}

\usepackage{float}
\usepackage{graphicx}
\usepackage{subcaption}
\usepackage{url}
\usepackage{color}
\usepackage{bm}
\usepackage{multirow}

\newtheorem{remark}{Remark}

\DeclareMathOperator*{\argmax}{arg\,max}

\title{Parallel-in-Time Solver for the All-at-Once Runge--Kutta Discretization}

\author{Santolo Leveque\thanks{CRM Ennio De Giorgi, Scuola Normale Superiore, Piazza dei Cavalieri 3, 56126, Pisa, Italy ({\tt santolo.leveque@sns.it})}
\and
Luca Bergamaschi\thanks{Department of Civil Environmental and Architectural Engineering, University of Padova, Via Marzolo 9, 35100, Padova, Italy (\tt luca.bergamaschi@unipd.it)}
\and
\'{A}ngeles Mart\'{i}nez\thanks{Department of Mathematics and Geosciences, University of Trieste, Via Valerio 12/1, 34127, Trieste, Italy (\tt amartinez@units.it)}
\and
John W. Pearson\thanks{School of Mathematics, The University of Edinburgh, James Clerk Maxwell Building, The King's Buildings, Peter Guthrie Tait Road, Edinburgh, EH9 3FD, United Kingdom ({\tt j.pearson@ed.ac.uk})}}

\begin{document}
\maketitle

\begin{abstract}
In this article, we derive fast and robust parallel-in-time preconditioned iterative methods for the all-at-once linear systems arising upon discretization of time-dependent PDEs. The discretization we employ is based on a Runge--Kutta method in time, for which the development of parallel solvers is an emerging research area in the literature of numerical methods for time-dependent PDEs. By making use of classical theory of block matrices, one is able to derive a preconditioner for the systems considered. The block structure of the preconditioner allows for parallelism in the time variable, as long as one is able to provide an optimal solver for the system of the stages of the method. We thus propose a preconditioner for the latter system based on a singular value decomposition (SVD) of the (real) Runge--Kutta matrix $A_{\mathrm{RK}} = U \Sigma V^\top$. Supposing $A_{\mathrm{RK}}$ is invertible, we prove that the spectrum of the system for the stages preconditioned by our SVD-based preconditioner is contained within the right-half of the unit circle, under suitable assumptions on the matrix $U^\top V$ (the assumptions are well posed due to the polar decomposition of $A_{\mathrm{RK}}$). We show the numerical efficiency of our SVD-based preconditioner by solving the system of the stages arising from the discretization of the heat equation and the Stokes equations, with sequential time-stepping. Finally, we provide numerical results of the all-at-once approach for both problems, showing the speed-up achieved on a parallel architecture.
\end{abstract}



\pagestyle{myheadings}
\thispagestyle{plain}
\markboth{S. LEVEQUE, L. BERGAMASCHI, \'{A}. MART\'{I}NEZ, AND J. W. PEARSON}{ALL-AT-ONCE SOLVER FOR RUNGE--KUTTA DISCRETIZATION}

\section{Introduction}\label{sec_1}
Time-dependent partial differential equations (PDEs) arise very often in the sciences, from mechanics to thermodynamics, from biology to economics, from engineering to chemistry, just to name a few. In fact, many physical processes can be described by the relation of some physical quantities using a differential operator. As problems involving (either steady or unsteady) PDEs usually lack a closed form solution, numerical methods are employed in order to find an approximation of it. These methods are based on discretizations of the quantities involved. For time-dependent PDEs, the discretization has also to take into account the time derivative. Classical numerical approaches employed to solve time-dependent PDEs result in a sequence of linear systems to be solved sequentially, mimicking the evolution in time of the physical quantities involved.

In the last few decades, many researchers have devoted their effort to devising parallel-in-time methods for the numerical solution of time-dependent PDEs, leading to the development of the Parareal \cite{Lions_Maday_Turinici}, the Parallel Full Approximation Scheme in Space and Time (PFASST) \cite{Emmett_Minion}, and the Multigrid Reduction in Time (MGRIT) \cite{Falgout_Friedhoff_Kolev_MacLachlan_Schroder, Friedhoff_Falgout_Kolev_MacLachlan_Schroder} algorithms, for example. As opposed to the classical approach, for which in order to obtain an approximation of the solution of an initial boundary value problem at a time $t$ one has to find an approximation of the solution at all the previous times, parallel-in-time methods approximate the solution of the problem for all times \emph{concurrently}. This in turns allows one to speed-up the convergence of the numerical solver by running the code on parallel architectures.

Among all the parallel-in-time approaches for solving time-dependent PDEs, increasing consideration has been given to the one introduced by Maday and R{\o}nquist in 2008 \cite{Maday_Ronquist}, see, for example, \cite{Gander_Halpern_Rannou_Ryan_2019, Liu_Wang_Wu_Zhou, McDonald_Pestana_Wathen}. This approach is based on a \emph{diagonalization} (ParaDiag) of the \emph{all-at-once} linear system arising upon the discretization of the differential operator. The diagonalization of the all-at-once system can be performed in two ways. First, by employing a non-constant time-step, for instance by employing a geometrically increasing sequence $\tau_1, \tau_2,\ldots,\tau_{n_t}$ as in \cite{Maday_Ronquist}, one can prove that the discretized system is diagonalizable \cite{Gander_Halpern_Rannou_Ryan_2016, Gander_Halpern_Rannou_Ryan_2019}. Second, one can employ a constant time-step and approximate the block-Toeplitz matrix arising upon discretization by employing, for instance, a circulant approximation \cite{McDonald_Pestana_Wathen}. Either way, the diagonalization of the all-at-once linear system under examination allows one to devise a preconditioner that can be run in parallel, obtaining thus a substantial speed-up.

Despite the efficiency and robustness of the ParaDiag preconditioner applied to the all-at-once discretization of the time-dependent PDE studied, this approach has a drawback. In fact, the discretization employed is based on linear multistep methods. It is well known that this class of methods are (in general) not A-stable, a property that allows one to choose an arbitrary time-step for the integration. Specifically, an A-stable linear multistep method cannot have order of convergence greater than two, as stated by the second Dahlquist barrier, see for example \cite[Theorem\,6.6]{Lambert}. In contrast, it is well known that one can devise an A-stable (implicit) Runge--Kutta method of any given order. Further, implicit Runge--Kutta methods have better stability properties than the linear multistep methods, (e.g., L- or B-stability, see for instance \cite{Butcher, Hairer_Wanner, Lambert}). For these reasons, in this work we present a (fully parallelizable) preconditioner for the all-at-once linear system arising when a Runge--Kutta method is employed for the time discretization. To the best of the authors' knowledge, this is the first attempt that fully focuses on deriving such a preconditioner of the form below for the all-at-once Runge--Kutta discretization. In this regard, we would like to mention the work \cite{Kressner_Massei_Zhu}, where the authors derived two solvers for the linear systems arising from an all-at-once approach of space-time discretization of time-dependent PDEs. The first solver is based on the observation that the numerical solution can be written as the sum of the solution of a system involving an $\alpha$-circulant matrix with the solution of a Sylvester equations with a right-hand side of low rank. The second approach is based on an interpolation strategy: the authors observe that the numerical solution can be approximated (under suitable assumptions) by a linear combination of the solutions of systems involving $\alpha$-circulant matrices, where $\alpha$ is the $j$th root of unity, for some integer $j$. Although the authors mainly focused on multistep methods, they adapted the two strategies in order to tackle also the all-at-once systems obtained when employing a Runge--Kutta method in time. We would like to note that the preconditioner derived in our work does not exploit the block-Toeplitz structure of the system arising upon discretization, therefore it could be employed also with a non-constant time-step.

We would like to mention that, despite the fact that one can obtain better stability properties when employing Runge--Kutta methods, this comes to a price. In fact, this class of method results in very large linear systems with a very complex structure: in order to derive an approximation of the solution at a time $t$, one has to solve a (sequence of) linear system(s) for the \emph{stages} of the discretization. Of late, a great effort has been devoted to devising preconditioners for the numerical solution for the stages of a Runge--Kutta method, see for example \cite{AbuLabdeh_MachLachlan_Farrell, Axelsson_Dravins_Neytcheva, Mardal_Nilssen_Staff, Rana_Howle_Long_Meek_Milestone, Southworth_Krzysik_Pazner_Sterck, Staff_Mardal_Nilssen}. As we will show below, the preconditioner for the all-at-once Runge--Kutta discretization results in a block-diagonal solve for all the stages of all the time-steps, and a Schur complement whose inverse can be applied by solving again for the systems for the stages of the method. Since the most expensive task is to (approximately) invert the system for the stages, in this work we also introduce a new block-preconditioner for the stage solver. This preconditioner is based on a SVD of the Runge--Kutta coefficient matrix, and has many advantages as we will describe below. Alternatively, one can also employ the strategies described in \cite{AbuLabdeh_MachLachlan_Farrell, Axelsson_Dravins_Neytcheva, Mardal_Nilssen_Staff, Southworth_Krzysik_Pazner_Sterck}, for example.

This paper is structured as follows. In Section \ref{sec_2}, we introduce Runge--Kutta methods and present the all-at-once system obtained upon discretization of a time-dependent differential equation by employing this class of methods in time. Then, we give specific details of the all-at-once system for the discretization of the heat equation and the Stokes equations. In Section \ref{sec_3}, we present the proposed preconditioner for the all-at-once system together with the approximation of the system of the stages of the Runge--Kutta method, for all the problems considered in this work. In Section \ref{sec_4}, we show the robustness and the parallel efficiency of the proposed preconditioning strategy. Finally, conclusions and future work are given in Section \ref{sec_5}.

\section{Runge--Kutta methods}\label{sec_2}
In this section, we present the linear systems arising upon discretization when employing a Runge--Kutta method in time. In what follows, $I_m$ represents the identity matrix of dimension $m$.

For simplicity, we integrate the ordinary differential equation $v'(t)=f(v(t),t)$ between $0$ and a final time $t_f>0$, given the initial condition $v(0)=v_0$. After dividing the time interval $[0,t_f]$ into $n_t$ subintervals with constant time-step $\tau$, the discretization of an $s$-stage Runge--Kutta method applied to $v'(t)=f(v(t),t)$ reads as follows:
\begin{displaymath}
v_{n+1}= v_{n} + \tau \sum_{i=1}^{s} b_i k_{i,n}, \quad n = 0, \ldots, n_t-1,
\end{displaymath}
where the stages $k_{i,n}$ are given by\footnote{Note that the stages $k_{i,n}$ for $i=1, \ldots,s$ represent an approximation of the time derivative of $v$.}
\begin{equation}\label{stages_Runge_Kutta}
k_{i,n} = f\left(v_n + \tau \sum_{j=1}^{s} a_{i,j} k_{j,n}, t_n + c_i \tau\right), \qquad i=1,\ldots,s,
\end{equation}
with $t_n=n \tau$. The Runge--Kutta method is uniquely defined by the coefficients $a_{i,j}$, the weights $b_{i}$, and the nodes $c_{i}$, for $i,j=1,\ldots,s$. For this reason, an $s$-stage Runge--Kutta method is defined by the following Butcher tableau:
\begin{equation*}\label{butcher_tableau}
\begin{array}{c|ccc}
c_1 & a_{1,1} & \ldots & a_{1,s}\\
\vdots & \vdots & \ddots & \vdots\\
c_s & a_{s,1} & \ldots & a_{s,s}\\
\hline
 & b_1 & \ldots & b_{s}
\end{array}
\end{equation*}
or in a more compact form
\begin{displaymath}
\def\arraystretch{1.2}
\begin{array}{c|c}
\mathbf{c}_{\mathrm{RK}} & A_{\mathrm{RK}}\\
\hline
 & \mathbf{b}_{\mathrm{RK}}^\top
\end{array}
\end{displaymath}
A Runge--Kutta method is said to be explicit if $a_{i,j}=0$ when $i \leq j$, otherwise it is called implicit. Note that for implicit Runge--Kutta method the stages are obtained by solving the non-linear equations \eqref{stages_Runge_Kutta}.

In the following, we will employ Runge--Kutta methods as the time discretization for the time-dependent differential equations considered in this work. In particular, we will focus here on the Stokes equations. Since this system of equations is properly a differential--algebraic equation (DAE), we cannot employ the Runge--Kutta method as we have described above. In order to fix the notation, given a domain $\Omega \subset \mathbb{R}^d$, with $d=1,2,3$, and a final time $t_f>0$, we consider the following DAE:
\begin{equation*}\label{general_DAE}
\left\{
\begin{array}{rl}
\vspace{0.25ex}
\frac{\partial {v}}{\partial t} + \mathcal{D}_1 v = {f}(\mathbf{x},t) & \quad \mathrm{in} \; \Omega \times (0,t_{f}), \\
\vspace{0.25ex}
\mathcal{D}_2 v = {g}(\mathbf{x},t) & \quad \mathrm{in} \; \Omega \times (0,t_{f}),
\end{array}
\right.
\end{equation*}
given some suitable initial and boundary conditions. Here, $\mathcal{D}_1$ and $\mathcal{D}_2$ are differential operators (only) in space. In addition, the variable $v$ may be a vector, and contains all the physical variables described by the DAE (e.g., the \emph{temperature} for the heat equation, or the \emph{velocity} of the fluid and the \emph{kinematic pressure} for the Stokes equations). For this reason, in our notation $\mathcal{D}_1$ and $\mathcal{D}_2$ as well as their discretizations have to be considered as ``vector'' differential operators. In what follows, we will suppose that the differential operators $\mathcal{D}_1$ and $\mathcal{D}_2$ are linear and time-independent.

Given suitable discretizations $\mathbf{D}_1$ and $\mathbf{D}_2$ of $\mathcal{D}_1$ and $\mathcal{D}_2$ respectively, after dividing the time interval $[0,t_f]$ into $n_t$ subintervals with constant time-step $\tau$, a Runge--Kutta discretization reads as follows:
\begin{equation}\label{RK_discretization_for_v}
\mathbf{M} \mathbf{v}_{n+1} = \mathbf{M} \mathbf{v}_n + \tau \mathbf{M} \sum_{i=1}^{s} b_i \mathbf{k}_{i,n}, \quad n = 0, \ldots, n_t-1,
\end{equation}
with a suitable discretization of the initial and boundary conditions. Here, $\mathbf{v}_{n}$ represents the discretization of $v$ at time $t_n$. The vectors $\mathbf{k}_{i,n}$ are defined as follows:
\begin{equation}\label{RK_stages}
\left\{
\begin{array}{ll}
\vspace{0.25ex}
\displaystyle \widehat{\mathbf{M}} \mathbf{k}_{i,n} + \mathbf{D}_1 \mathbf{v}_n + \tau \mathbf{D}_1 \sum_{j=1}^{s} a_{i,j} \mathbf{k}_{j,n} = \mathbf{f}_{i,n},& \quad i=1,\ldots,s,\\
\displaystyle \mathbf{D}_2 \mathbf{v}_n + \tau \mathbf{D}_2 \sum_{j=1}^{s} a_{i,j} \mathbf{k}_{j,n} = \mathbf{g}_{i,n},& \quad i=1,\ldots,s,
\end{array}
\right.
\end{equation}
for $n=0,\ldots,n_t-1$, where $\mathbf{f}_{i,n}$ and $\mathbf{g}_{i,n}$ are discretizations of the functions $f$ and $g$ at the time $t_n+ c_i \tau$, respectively. Finally, the matrices $\mathbf{M}$ and $\widehat{\mathbf{M}}$ are suitable discretizations of the identity operator, with $\mathbf{M}$ defined over the whole space the variable $v(\cdot,t_n)$ belongs to, while $\widehat{\mathbf{M}}$ may be defined over only a subset of the latter space. We recall that the stages $\mathbf{k}_{i,n}$ for $i=1, \ldots,s$ represent an approximation of $\frac{\partial v}{\partial t}$; therefore, the (Dirichlet) boundary conditions on $\mathbf{k}_{i,n}$ are given by the time derivatives of the corresponding boundary conditions on $v$. Note that, if we relax the assumption of $\mathcal{D}_1$ and $\mathcal{D}_2$ being time-independent, \eqref{RK_stages} would be properly written as follows:
\begin{displaymath}
\left\{
\begin{array}{ll}
\vspace{0.25ex}
\displaystyle \widehat{\mathbf{M}} \mathbf{k}_{i,n} + \mathbf{D}_1^{(n,i)} \mathbf{v}_n + \tau \mathbf{D}_1^{(n,i)} \sum_{j=1}^{s} a_{i,j} \mathbf{k}_{j,n} = \mathbf{f}_{i,n},& \quad i=1,\ldots,s,\\
\displaystyle \mathbf{D}_2^{(n,i)} \mathbf{v}_n + \tau \mathbf{D}_2^{(n,i)} \sum_{j=1}^{s} a_{i,j} \mathbf{k}_{j,n} = \mathbf{g}_{i,n},& \quad i=1,\ldots,s,
\end{array}
\right.
\end{displaymath}
with $\mathbf{D}_1^{(n,i)}$ and $\mathbf{D}_2^{(n,i)}$ the discretizations of the differential operators $\mathcal{D}_1$ and $\mathcal{D}_2$ at time $t_n+c_i\tau$, respectively, for $n=0,\ldots,n_t-1$.

In compact form, we can rewrite \eqref{RK_stages} as follows:
\begin{displaymath}
\left\{
\begin{array}{l}
\vspace{0.25ex}
(I_s \otimes \widehat{\mathbf{M}}) \mathbf{k}_{n} + (\mathbf{e} \otimes \mathbf{D}_1) \mathbf{v}_n + \tau (A_{\mathrm{RK}} \otimes \mathbf{D}_1) \mathbf{k}_{n} = \mathbf{f}_{n},\\
(\mathbf{e} \otimes \mathbf{D}_2) \mathbf{v}_n + \tau (A_{\mathrm{RK}} \otimes \mathbf{D}_2) \mathbf{k}_{n} = \mathbf{g}_{n},
\end{array}
\right.
\end{displaymath}
where $\mathbf{e}\in \mathbb{R}^s$ is the column vector of all ones. Here, we set $\mathbf{k}_{n}=[\mathbf{k}_{1,n}^\top,\ldots,\mathbf{k}_{s,n}^\top]^\top$, $\mathbf{f}_{n}=[\mathbf{f}_{1,n}^\top,\ldots,\mathbf{f}_{s,n}^\top]^\top$, and $\mathbf{g}_{n}=[\mathbf{g}_{1,n}^\top,\ldots,\mathbf{g}_{s,n}^\top]^\top$. Further, we may rewrite \eqref{RK_discretization_for_v} as
\begin{displaymath}
\mathbf{M} \mathbf{v}_{n+1} = \mathbf{M} \mathbf{v}_n + \tau ( \mathbf{b}_{\mathrm{RK}}^\top \otimes \mathbf{M} ) \mathbf{k}_{n}.
\end{displaymath}

We are now able to write the all-at-once system for the Runge--Kutta discretization in time. By setting $\mathbf{v}=[\mathbf{v}_{0}^\top,\ldots, \mathbf{v}_{n_t}^\top]^\top$ and $\mathbf{k}=[\mathbf{k}_{0}^\top,\ldots, \mathbf{k}_{n_t-1}^\top]^\top$, we can rewrite \eqref{RK_discretization_for_v}--\eqref{RK_stages} in matrix form as
\begin{equation}\label{system_all_at_once_RK}
\underbrace{\left[
\begin{array}{cc}
\Phi & \Psi_1\\
\Psi_2 & \Theta
\end{array}
\right]}_{\mathcal{A}}
\left[
\begin{array}{c}
\mathbf{v}\\
\mathbf{k}
\end{array}
\right]=
\mathbf{b}.
\end{equation}
The blocks of the matrix $\mathcal{A}$ are given by
\begin{equation}\label{blocks_all_at_once_RK}
\begin{array}{ll}
\vspace{1ex}
\!\!
\Phi \! = \! \left[
\!\!
\begin{array}{cccc}
\mathbf{M}\\
-\mathbf{M} & \ddots\\
 & \ddots & \ddots\\
 & & -\mathbf{M} & \mathbf{M}
\end{array}
\!\!
\right] \!,
&
\!\!
\Psi_1 \! = \! - \! \left[
\!\!
\begin{array}{ccc}
0\\
\tau \mathbf{b}_{\mathrm{RK}}^\top \otimes \mathbf{M} \\
 & \ddots \\
 & & \tau \mathbf{b}_{\mathrm{RK}}^\top \otimes \mathbf{M}
\end{array}
\!\!
\right] \! ,\\
\!\!
\Psi_2 \! = \! \left[
\!\!
\begin{array}{cccc}
\mathbf{e} \otimes \mathbf{D}_1\\
\mathbf{e} \otimes \mathbf{D}_2\\
 & \ddots\\
 & & \mathbf{e} \otimes \mathbf{D}_1 & 0\\
 & & \mathbf{e} \otimes \mathbf{D}_2 & 0\\
\end{array}
\!\!
\right] \!,
&
\!\!
\Theta \! = \! I_{n_t} \! \otimes \!
\left[
\!\!
\begin{array}{c}
I_s \otimes \widehat{\mathbf{M}} + \tau A_{\mathrm{RK}}\otimes \mathbf{D}_1\\
\tau A_{\mathrm{RK}}\otimes \mathbf{D}_2
\end{array}
\!\!
\right] \!.
\end{array}
\end{equation}
Note that $\Theta$ is block-diagonal.

In what follows, we will present the all-at-once system \eqref{system_all_at_once_RK} for the specific cases of the heat equation and the Stokes equations.

\subsection{Heat equation}\label{sec_2_1}
Given a domain $\Omega \subset \mathbb{R}^d$, with $d=1,2,3$, and a final time $t_f>0$, we consider the following heat equation:
\begin{equation}\label{heat_equation}
\left\{
\begin{array}{rl}
\vspace{0.25ex}
\frac{\partial {v}}{\partial t} - \nabla^2 {v} = {f}(\mathbf{x},t) & \quad \mathrm{in} \; \Omega \times (0,t_{f}), \\
\vspace{0.25ex}
{v}(\mathbf{x},t) = {g}(\mathbf{x},t) & \quad \mathrm{on} \; \partial \Omega \times (0,t_{f}),\\
{v}(\mathbf{x},0) = {v}_0(\mathbf{x}) & \quad \mathrm{in} \; \Omega,

\end{array}
\right.
\end{equation}
where the functions ${f}$ and ${g}$ are known. In addition, the initial condition ${v}_0(\mathbf{x})$ is also given.

After dividing the time interval $[0,t_f]$ into $n_t$ subintervals, the discretization of \eqref{heat_equation} by employing a Runge--Kutta method reads as follows:
\begin{equation}\label{RK_discretization_for_v_heat_equation}
M \mathbf{v}_{n+1} = M \mathbf{v}_n + \tau M \sum_{i=1}^{s} b_i \mathbf{k}_{i,n}, \quad n = 0, \ldots, n_t-1,
\end{equation}
with $M \mathbf{v}_0 = M \mathbf{v}^0$ a suitable discretization of the initial condition. The stages $\mathbf{k}_{i,n}$ are defined as follows:
\begin{equation}\label{RK_stages_heat_equation}
\begin{array}{ll}
M \mathbf{k}_{i,n} + K \mathbf{v}_n + \tau K \sum_{j=1}^{s} a_{i,j} \mathbf{k}_{j,n} = \mathbf{f}_{i,n}, & \:\: i=1,\ldots,s, \;\;\: n=0,\ldots,n_t-1,
\end{array}
\end{equation}
where
\begin{displaymath}
(\mathbf{f}_{i,n})_m = \int_{\Omega} f(\mathbf{x}, t_n + c_i \tau) \phi_m \: \mathrm{d}\Omega, \quad i = 1, \ldots, s.
\end{displaymath}
Here, $K$ and $M$ are the \emph{stiffness} and \emph{mass} matrices, respectively. For a Dirichlet problem, both matrices are symmetric positive definite (s.p.d.). As we mentioned above, the boundary conditions on the stages $\mathbf{k}_{i,n}$ are given by the time derivatives of the corresponding boundary conditions on $v$. Specifically, we have
\begin{displaymath}
\begin{array}{c}
\left. (\mathbf{k}_{i,n}) \right| _{\partial \Omega} = \frac{\partial g}{\partial t} (\cdot, t_n + c_i \tau).
\end{array}
\end{displaymath}

By adopting an all-at-once approach, we can rewrite the system \eqref{RK_discretization_for_v_heat_equation}--\eqref{RK_stages_heat_equation} as follows:
\begin{displaymath}
\left\{
\begin{array}{ll}
M \mathbf{v}_0 = M \mathbf{v}^0,\\
M \mathbf{v}_{n+1} - M \mathbf{v}_n - \tau M \sum_{j=1}^{s} b_j \mathbf{k}_{j,n} = \mathbf{0}, & \quad n = 0, \ldots, n_t-1,\\
M \mathbf{k}_{i,n} + K \mathbf{v}_n + \tau K \sum_{j=1}^{s} a_{i,j} \mathbf{k}_{j,n} = \mathbf{f}_{i,n}, & \quad i=1,\ldots,s, \;\;\: n=0,\ldots,n_t-1.
\end{array}
\right.
\end{displaymath}
In matrix form, we have
\begin{equation}\label{system_all_at_once_RK_heat_equation}
\underbrace{\left[
\begin{array}{cc}
\Phi & \Psi_1\\
\Psi_2 & \Theta
\end{array}
\right]}_{\mathcal{A}}
\left[
\begin{array}{c}
\mathbf{v}_0\\
\mathbf{v}_1\\
\vdots\\
\mathbf{v}_{n_t}\\
\mathbf{k}_0\\
\vdots\\
\mathbf{k}_{n_t-1}
\end{array}
\right]=
\left[
\begin{array}{c}
\mathbf{v}^0\\
\mathbf{b}_1\\
\vdots\\
\mathbf{b}_{n_t}\\
\mathbf{f}_0\\
\vdots\\
\mathbf{f}_{n_t-1}
\end{array}
\right],
\end{equation}
where the vectors $\mathbf{b}_n$, $n=1,\ldots,n_t$, contain information about the boundary conditions. The blocks of the matrix $\mathcal{A}$ are given by
\begin{equation*}\label{blocks_all_at_once_RK_heat_equation}
\begin{array}{ll}
\vspace{1ex}
\!\!
\Phi \! = \! \left[
\!\!
\begin{array}{cccc}
M\\
-M & \ddots\\
 & \ddots & \ddots\\
 & & -M & M
\end{array}
\!\!
\right] \!,
&
\!\!
\Psi_1 \! = \! - \! \left[
\!\!
\begin{array}{ccc}
0\\
\tau \mathbf{b}_{\mathrm{RK}}^\top \otimes M \\
 & \ddots \\
 & & \tau \mathbf{b}_{\mathrm{RK}}^\top \otimes M
\end{array}
\!\!
\right] \!,\\
\!\!
\Psi_2 \! = \! \left[
\!\!
\begin{array}{cccc}
\mathbf{e} \otimes K\\
 & \ddots\\
 & & \mathbf{e} \otimes K & 0
\end{array}
\!\!
\right] \!,
&
\!\!
\Theta = I_{n_t} \otimes (I_s \otimes M + \tau A_{\mathrm{RK}}\otimes K) .
\end{array}
\end{equation*}

\subsection{Stokes equations}\label{sec_2_2}
Given a domain $\Omega \subset \mathbb{R}^d$, with $d=2,3$, and a final time $t_f>0$, we consider the following Stokes equations:
\begin{equation}\label{Stokes_equation}
\left\{
\begin{array}{rl}
\vspace{0.25ex}
\frac{\partial \vec{v}}{\partial t} - \nabla^2 \vec{v} + \nabla {p} = \vec{f}(\mathbf{x},t) & \quad \mathrm{in} \; \Omega \times (0,t_{f}), \\
\vspace{0.25ex}
- \nabla \cdot \vec{v} = 0 & \quad \mathrm{in} \; \Omega \times (0,t_{f}), \\
\vspace{0.25ex}
\vec{v}(\mathbf{x},t) = \vec{g}(\mathbf{x},t) & \quad \mathrm{on} \; \partial \Omega \times (0,t_{f}),\\
\vec{v}(\mathbf{x},0) = \vec{v}_0(\mathbf{x}) & \quad \mathrm{in} \; \Omega.

\end{array}
\right.
\end{equation}
As above, the functions $\vec{f}$ and $\vec{g}$ as well as the initial condition $\vec{v}_0(\mathbf{x})$ are known.

After dividing the time interval $[0,t_f]$ into $n_t$ subintervals, the discretization of \eqref{Stokes_equation} by a Runge--Kutta method reads as follows:
\begin{equation}\label{RK_discretization_for_vp_Stokes_equations}
\begin{array}{ll}
\vspace{0.25ex}
M_v \mathbf{v}_{n+1} = M_v \mathbf{v}_n + \tau M_v \sum_{i=1}^{s} b_i \mathbf{k}^{v}_{i,n}, & \quad n = 0, \ldots, n_t-1,\\
M_p \mathbf{p}_{n+1} = M_p \mathbf{p}_n + \tau M_p \sum_{i=1}^{s} b_i \mathbf{k}^{p}_{i,n}, & \quad n = 0, \ldots, n_t-1,
\end{array}
\end{equation}
with $M_v \mathbf{v}_0 = M_v \mathbf{v}^0$ a suitable discretization of the initial condition, and $M_p \mathbf{p}_0 = M_p \mathbf{p}^0$ a suitable approximation of the pressure $p$ at time $t=0$. The stages $\mathbf{k}^{v}_{i,n}$ and $\mathbf{k}^{p}_{i,n}$ are defined as follows:
\begin{equation}\label{RK_stages_Stokes_equations}
\begin{array}{ll}
\vspace{0.25ex}
M_v \mathbf{k}^{v}_{i,n} + K_v \mathbf{v}_n + \tau K_v \sum_{j=1}^{s} a_{i,j} \mathbf{k}^{v}_{j,n} + B^\top \mathbf{p}_n + \tau B^\top \sum_{j=1}^{s} a_{i,j} \mathbf{k}^{p}_{j,n} = \mathbf{f}_{i,n}, \\
B \mathbf{v}_n + \tau B \sum_{j=1}^{s} a_{i,j} \mathbf{k}^{v}_{j,n} = \mathbf{0},
\end{array}
\end{equation}
for $i=1,\ldots,s$ and $n=0,\ldots,n_t-1$, where
\begin{displaymath}
(\mathbf{f}_{i,n})_m = \int_{\Omega} \vec{f}(\mathbf{x}, t_n + c_i \tau) \cdot \vec{\phi}_m \: \mathrm{d}\Omega, \quad i = 1, \ldots, s.
\end{displaymath}
Here, $K_v$ and $M_v$ (resp., $K_p$ and $M_p$) are the \emph{vector-stiffness} and \emph{vector-mass} matrices (resp., stiffness and mass), respectively. Finally, as above, the boundary conditions on $\mathbf{k}^{v}_{i,n}$ are given by
\begin{displaymath}
\begin{array}{c}
\left.(\mathbf{k}^{v}_{i,n})\right| _{\partial \Omega} = \frac{\partial \vec{g}}{\partial t} (\cdot, t_n + c_i \tau).
\end{array}
\end{displaymath}

By adopting an all-at-once approach, we can rewrite the system \eqref{RK_discretization_for_vp_Stokes_equations}--\eqref{RK_stages_Stokes_equations} as follows:
\begin{displaymath}
\left\{
\begin{array}{ll}
\!\! M_v \mathbf{v}_0 = M_v \mathbf{v}^0,\\
\!\! M_p \mathbf{p}_0 = M_p \mathbf{p}^0,\\
\!\! M_v \mathbf{v}_{n+1} - M_v \mathbf{v}_n - \tau M_v \sum_{j=1}^{s} b_j \mathbf{k}^{v}_{j,n} = \mathbf{0}, & \\
\!\! M_p \mathbf{p}_{n+1} - M_p \mathbf{p}_n - \tau M_p \sum_{i=1}^{s} b_i \mathbf{k}^{p}_{i,n} = \mathbf{0}, & \\
\!\! M_v \mathbf{k}^{v}_{i,n} \! + \! K_v (\mathbf{v}_n \! + \! \tau \sum_{j=1}^{s} a_{i,j} \mathbf{k}^{v}_{j,n}) \! + \! B^\top  \! (\mathbf{p}_n \! + \! \tau \sum_{j=1}^{s} a_{i,j} \mathbf{k}^{p}_{j,n}) \! = \! \mathbf{f}_{i,n}, & \! i=1,\ldots,s,\\
\!\! B \mathbf{v}_n + \tau B \sum_{j=1}^{s} a_{i,j} \mathbf{k}^{v}_{j,n} = \mathbf{0}, & \! i=1,\ldots,s,
\end{array}
\right.
\end{displaymath}
for $n=0,\ldots,n_t-1$. In matrix form, the system is of the form \eqref{system_all_at_once_RK}, with
\begin{displaymath}
\begin{array}{l}
\vspace{0.3ex}
\mathbf{v}=[\mathbf{v}_0^\top,\mathbf{p}_0^\top, \ldots, \mathbf{v}_{n_t}^\top,\mathbf{p}_{n_t}^\top]^\top,\\
\mathbf{k}=[(\mathbf{k}^v_{0})^\top, (\mathbf{k}^p_{0})^\top, \ldots, (\mathbf{k}^v_{n_t-1})^\top, \ldots, (\mathbf{k}^p_{n_t-1})^\top]^\top,
\end{array}
\end{displaymath}
where
\begin{displaymath}
\begin{array}{c}
\mathbf{k}^v_{n} = [(\mathbf{k}^v_{n,1})^\top, \ldots, (\mathbf{k}^v_{n,s})^\top]^\top, \quad \mathbf{k}^p_{n} = [(\mathbf{k}^p_{n,1})^\top, \ldots, (\mathbf{k}^p_{n,s})^\top]^\top, \quad n=0, \ldots, n_t-1.
\end{array}
\end{displaymath}
Further, the blocks of the matrix $\mathcal{A}$ are as in \eqref{blocks_all_at_once_RK}, with $\Theta = I_{n_t} \otimes \widehat{\Theta}$, and
\begin{displaymath}
\begin{array}{ll}
\Psi_2 = \left[
\begin{array}{cccc}
\widehat{\Psi}_2\\
 & \ddots\\
 & & \widehat{\Psi}_2 & 0
\end{array}
\right],
&
\quad
\mathbf{M} = \left[
\begin{array}{cc}
M_v & 0\\
0 & M_p
\end{array}
\right].
\end{array}
\end{displaymath}
Here, the blocks are given by
\begin{equation}\label{system_for_stages_Stokes}
\begin{array}{cc}
\!\!
\widehat{\Psi}_2 \! = \! \left[
\!\!
\begin{array}{cc}
\mathbf{e} \otimes K_v & \mathbf{e} \otimes B^\top\\
\mathbf{e} \otimes B & 0
\end{array}
\!\!
\right], &
\!\!
\widehat{\Theta} \! = \! \left[
\!\!
\begin{array}{cc}
I_s \otimes M_v + \tau A_{\mathrm{RK}}\otimes K_v & \tau A_{\mathrm{RK}}\otimes B^\top\\
\tau A_{\mathrm{RK}}\otimes B & 0
\end{array}
\!\!
\right].
\end{array}
\end{equation}
Note that we need an approximation of the pressure $p$ at time $t=0$. In our tests, before solving for the all-at-once system we integrate the problem between $(0, \epsilon_{\mathrm{BE}})$ employing one step of backward Euler. The approximations of the velocity $\vec{v}$ and pressure $p$ at time $t=\epsilon_\mathrm{BE}$ are then employed as initial conditions for solving the problem in $(\epsilon_\mathrm{BE},t_f)$. In our tests, we choose $\epsilon_{\mathrm{BE}}=h^{2.5}$, with $h$ the mesh size in space. This choice has been made in order to achieve an accurate enough solution at time $t=\epsilon_{\mathrm{BE}}$ and a fast solver for the backward Euler discretization. We would like to mention that finding suitable initial conditions for the problem \eqref{RK_discretization_for_vp_Stokes_equations}--\eqref{RK_stages_Stokes_equations} is beyond the scope of this work, and one can employ other approaches. For instance, one can employ a \emph{solenoidal projection}, as done in \cite{Hinze_Koster_Turek}.

\section{Preconditioner}\label{sec_3}
In what follows, we denote with $\sigma(\cdot)$ the spectrum of a given matrix.

Supposing that $\Theta$ is invertible, we consider as a preconditioner for the system \eqref{system_all_at_once_RK} the following matrix:
\begin{equation}\label{preconditioner_all_at_once}
\mathcal{P}=\left[
\begin{array}{cc}
S & \Psi_1\\
0 & \Theta
\end{array}
\right],
\end{equation}
where $S=-\Phi-\Psi_1 \Theta^{-1}\Psi_2$ is the Schur complement, with $\Phi$, $\Psi_1$, $\Psi_2$, and $\Theta$ defined as in \eqref{blocks_all_at_once_RK}. Specifically, we have
\begin{displaymath}
S=-\left[
\begin{array}{cccc}
\mathbf{M}\\
-\mathbf{M}+X & \ddots\\
& \ddots & \ddots\\
& & -\mathbf{M}+X & \mathbf{M}
\end{array}
\right],
\end{displaymath}
where
\begin{displaymath}
X= \tau \left[
\begin{array}{ccc}
b_1 \mathbf{M} & \ldots & b_s \mathbf{M}\\
\end{array}
\right] \left[
\begin{array}{c}
I_s \otimes \widehat{\mathbf{M}} + \tau A_{\mathrm{RK}}\otimes \mathbf{D}_1\\
\tau A_{\mathrm{RK}}\otimes \mathbf{D}_2
\end{array}
\right]^{-1}
\left[
\begin{array}{c}
\mathbf{e} \otimes \mathbf{D}_1\\
\mathbf{e} \otimes \mathbf{D}_2
\end{array}
\right].
\end{displaymath}
The preconditioner $\mathcal{P}$ given in \eqref{preconditioner_all_at_once} is optimal. In fact, supposing also that $S$ is invertible, one can prove that $\sigma(\mathcal{P}^{-1}\mathcal{A})=\left\{1 \right\}$, and the minimal polynomial of the preconditioned matrix has degree 2, see, for instance, \cite{Ipsen01, Murphy:1999:NPI:359189.359190}. For this reason, when employing the preconditioner $\mathcal{P}$, an appropriate iterative method should converge in at most two iterations (in exact arithmetic). However, in practical applications even forming the Schur complement $S$ may be unfeasible due to the large dimensions of the system. Besides, the block $\Theta$ may be singular, in which case not only is the Schur complement $S$ not well defined, but also we cannot apply the inverse of $\mathcal{P}$. For this reason, rather than solving for the matrix $\mathcal{P}$, we favour finding a cheap invertible approximation $\widetilde{\mathcal{P}}$ of $\mathcal{P}$, in which the block $\Theta$ is replaced by an invertible $\widetilde{\Theta}$ and the Schur complement $S$ is approximated by $\widetilde{S}\approx -\Phi-\Psi_1 \widetilde{\Theta}^{-1}\Psi_2$, respectively. In what follows, we will find approximations of the main blocks of $\mathcal{P}$, focusing mainly on parallelizable approximations.

Clearly, the matrix $\Theta$ defined in \eqref{blocks_all_at_once_RK} is block-diagonal, with each diagonal block given by the system for the stages. Therefore, a cheap method for approximately inverting the system for the stages gives also a cheap way for approximately inverting the block $\Theta$. Note that the latter may be performed in parallel over all the time steps $n_t$.

We now focus on an approximation of the Schur complement $S$. The latter may be factorized as follows:
\begin{equation}\label{widehat_S}
S=-\left[
\begin{array}{cccc}
\mathbf{M}\\
 & \ddots\\
 &  & \mathbf{M}
\end{array}
\right]
\underbrace{\left[
\begin{array}{cccc}
I_{n_x}\\
-I_{n_x}+\widehat{X} & \ddots\\
& \ddots & \ddots\\
& & -I_{n_x}+\widehat{X} & I_{n_x}
\end{array}
\right]}_{\widehat{S}},
\end{equation}
where
\begin{equation}\label{widehat_X}
\widehat{X}= \tau \left[
\begin{array}{ccc}
b_1 I_{n_x} & \ldots & b_s I_{n_x}\\
\end{array}
\right] \left[
\begin{array}{c}
I_s \otimes \widehat{\mathbf{M}} + \tau A_{\mathrm{RK}}\otimes \mathbf{D}_1\\
\tau A_{\mathrm{RK}}\otimes \mathbf{D}_2
\end{array}
\right]^{-1}
\left[
\begin{array}{c}
\mathbf{e} \otimes \mathbf{D}_1\\
\mathbf{e} \otimes \mathbf{D}_2
\end{array}
\right].
\end{equation}
A parallel solve for $S$ may be performed if we have a parallel solve for $\widehat{S}$. The latter may be done by employing an MGRIT routine \cite{Falgout_Friedhoff_Kolev_MacLachlan_Schroder, Friedhoff_Falgout_Kolev_MacLachlan_Schroder}, for example. In our tests, we employ the XBraid v3.0.0 routine \cite{XBraid}.

As we mentioned, the main computational task is to (approximately) solve for the linear system of the stages. In the following, we present the strategy adopted in this work. Again, we would like to mention that one may also employ other solvers for the stages, providing their optimality.

In what follows, we will assume that the matrix $A_{\mathrm{RK}}$ is invertible.

\subsection{Preconditioner for the stages}\label{sec_3_1}
As we mentioned above, an all-at-once solve for a Runge--Kutta discretization in time may be performed only if one has an optimal preconditioner for the system of the stages
\begin{displaymath}
\underbrace{\left[
\begin{array}{c}
I_s \otimes \widehat{\mathbf{M}} + \tau A_{\mathrm{RK}}\otimes \mathbf{D}_1\\
\tau A_{\mathrm{RK}}\otimes \mathbf{D}_2
\end{array}
\right]}_{\widehat{\Theta}}
\left[
\begin{array}{c}
\mathbf{k}_{1,n}\\
\vdots\\
\mathbf{k}_{s,n}
\end{array}
\right]=
\left[
\begin{array}{c}
\mathbf{b}_{1}\\
\vdots\\
\mathbf{b}_{s}
\end{array}
\right].
\end{displaymath}

In order to derive a preconditioner for the matrix $\widehat{\Theta}$, we consider a SVD decomposition of the matrix $A_{\mathrm{RK}}=U \Sigma V^*$, where $U$ and $V$ are unitary matrices whose columns are the left and right singular vectors of $A_{\mathrm{RK}}$, respectively, and $\Sigma$ is a diagonal matrix with entries equal to the singular values of $A_{\mathrm{RK}}$. Note that this decomposition is not unique. Note also that since the matrix $A_{\mathrm{RK}}$ is real, the matrices $U$ and $V$ can be chosen to be real \cite[Section\,2.4]{Golub_van_Loan}, therefore they are properly orthogonal matrices. For this reason, we can write $A_{\mathrm{RK}}=U \Sigma V^\top$. From here, we can write
\begin{align*}
I_s \otimes \widehat{\mathbf{M}} + \tau A_{\mathrm{RK}}\otimes \mathbf{D}_1 &  = I_s \otimes \widehat{\mathbf{M}} + \tau (U \Sigma V^\top) \otimes \mathbf{D}_1\\
& = ( U \otimes I_{\hat{n}_x} ) [ (U^\top V) \otimes \widehat{\mathbf{M}} + \tau \Sigma \otimes \mathbf{D}_1 ]
\!
\left[
\!\!
\begin{array}{cc}
V^\top \otimes I_{\hat{n}_x} & \!\!\! 0 \\
0 & \!\!\! V^\top \otimes I_{\hat{m}_x}
\end{array}
\!\!\!
\right],
\end{align*}
with $\hat{n}_x+\hat{m}_x=n_x$. Note that, since the matrices $U$ and $V$ are orthogonal, the same holds for the matrix $U^\top V$. In particular, the eigenvalues of the matrix $U^\top V$ all lie on the unit circle centered at the origin of the complex plane, and its eigenvectors are mutually orthogonal. Since the eigenvalues have all absolute value equal to $1$, we can derive the following approximation:
\begin{displaymath}
\mathcal{P}_{\mathrm{RK}}:=
\left[
\!\!
\begin{array}{cc}
U \otimes I_{\hat{n}_x} & \!\! 0 \\
0 & \!\! U \otimes I_{\hat{m}_x} 
\end{array}
\!\!
\right]
\left[
\!\!
\begin{array}{c}
I_s \otimes \widehat{\mathbf{M}} + \tau \Sigma \otimes \mathbf{D}_1\\
\tau \Sigma \otimes \mathbf{D}_2
\end{array}
\!\!
\right]
\left[
\!\!
\begin{array}{cc}
V^\top \otimes I_{\hat{n}_x} & \!\! 0 \\
0 & \!\! V^\top \otimes I_{\hat{m}_x}
\end{array}
\!\!\!
\right] \approx \widehat{\Theta}.
\end{displaymath}
This approximation can be employed as a preconditioner for the matrix $\widehat{\Theta}$ within the GMRES algorithm derived in \cite{Saad_Schultz}. Note that, excluding the effect of the inverses of the matrices $U \otimes I_{\hat{n}_x}$, $U \otimes I_{\hat{m}_x}$, $V^\top \otimes I_{\hat{n}_x}$, and $V^\top \otimes I_{\hat{m}_x}$ (which require only a matrix--vector product), the $(1,1)$-block of the preconditioner is block-diagonal, therefore its application can be done in parallel (we will discuss in more detail how to deal with the other blocks of the preconditioner in the case of the Stokes equations). Further, the matrices are all real, and we are not forced to work in complex arithmetic. Finally, we would like to mention that, compared to an eigendecomposition, by employing this strategy one is able to avoid possibly ill-conditioned matrices arising from the eigenvectors, for example. Despite the above properties holding, one cannot expect the approximation $\mathcal{P}_{\mathrm{RK}}$ to be completely robust. In fact, the matrix $U^\top V$ has $s$ distinct eigenvalues, therefore we expect the preconditioned matrix $\mathcal{P}_{\mathrm{RK}}^{-1}\widehat{\Theta}$ to have $s$ clusters of eigenvalues. Nonetheless, we expect the preconditioner to work robustly at least with respect to the mesh size $h$.

Below, we will show how to employ the preconditioner $\mathcal{P}_{\mathrm{RK}}$ for solving for the matrix $\widehat{\Theta}$.

\subsubsection{Heat equation}\label{sec_3_1_1}
Before specifying our strategy for the heat equation, we would like to introduce other preconditioners employed for solving for the stages of a Runge--Kutta discretization for this problem.

In \cite{Mardal_Nilssen_Staff}, the authors approximate the matrix $I_s \otimes M + \tau A_{\mathrm{RK}}\otimes K$ with a fixed number of GMRES iteration, preconditioned with the following matrix:
\begin{displaymath}
\mathcal{P}_{\mathrm{MNS}}=
\left[
\begin{array}{ccc}
M + \tau a_{1,1} K \\
& \ddots\\
& & M+\tau a_{s,s}K
\end{array}
\right].
\end{displaymath}
The preconditioner $\mathcal{P}_{\mathrm{MNS}}$ is optimal, in the sense that it can be proved that the condition number of the preconditioned system $\mathcal{P}_{\mathrm{MNS}}^{-1}(I_s \otimes M + \tau A_{\mathrm{RK}}\otimes K)$ is independent of the time-step $\tau$ and the mesh size $h$, see \cite{Mardal_Nilssen_Staff}. However, numerical experiments show that the condition number may be dependent on the number of stages $s$, see \cite{Mardal_Nilssen_Staff}. We would like to mention that other approaches may be employed, since the parallel solver we propose for the all-at-once system is mainly based on a solver for the system of the stages of a Runge--Kutta method. For instance, in \cite{Staff_Mardal_Nilssen} the authors employ as a preconditioner the block-lower triangular part of the matrix $I_s \otimes M + \tau A_{\mathrm{RK}}\otimes K$, obtaining more robustness with respect to the number of stages $s$. Alternatively, one may employ the strategies described, for instance, in \cite{AbuLabdeh_MachLachlan_Farrell, Axelsson_Dravins_Neytcheva, Rana_Howle_Long_Meek_Milestone} as a preconditioner for the linear system considered.

In the numerical tests below, we compare our preconditioner $\mathcal{P}_{\mathrm{RK}}$ only with the preconditioner $\mathcal{P}_{\mathrm{MNS}}$. This is done for various reasons. In fact, although the methods presented in \cite{Staff_Mardal_Nilssen, Rana_Howle_Long_Meek_Milestone} are robust, the preconditioners are not fully parallelizable as they are presented, as one needs to solve for a block-lower triangular matrix. On the other hand, as discussed in \cite{Axelsson_Dravins_Neytcheva}, a preconditioner based on the diagonalization of the matrix $A_{\mathrm{RK}}$ makes use of complex arithmetic, therefore, although completely parallelizable, the strategy requires one to solve for systems twice the dimension of each block in order to work with real arithmetic.

We can now describe the preconditioner $\mathcal{P}_{\mathrm{RK}}$ employed for solving for the stages of the discretization of the heat equation. The system for the stages is given by
\begin{displaymath}
I_s \otimes M + \tau A_{\mathrm{RK}}\otimes K.
\end{displaymath}
In order to solve for this matrix, we employ GMRES with the following preconditioner:
\begin{equation}\label{prec_heat_eqaution}
\mathcal{P}_{\mathrm{RK}}= ( U \otimes I_{n_x} ) (I_s \otimes M + \tau \Sigma \otimes K) ( V^\top \otimes I_{n_x} ).
\end{equation}

The following theorem provides the location of the eigenvalues of the preconditioned system, under an assumption which frequently holds. More specifically, we require that the real part of the Rayleigh quotient $\frac{\mathbf{x}^* (U^\top V)\mathbf{x}}{\mathbf{x}^* \mathbf{x}}$ is positive, for any $\mathbf{x}\in \mathbb{C}^s$ with $\mathbf{x}\neq \mathbf{0}$. Note that, since the matrix $U^\top V$ is orthogonal (in particular, it is normal), this is equivalent to saying that the real part of the eigenvalues of the matrix $U^\top V$ is positive, since the set $\left\{\frac{\mathbf{x}^* (U^\top V)\mathbf{x}}{\mathbf{x}^* \mathbf{x}} | \mathbf{x}\in \mathbb{C}^s, \, \mathbf{x}\neq \mathbf{0} \right\}$ describes the \emph{field of values} of the matrix $U^\top V$, and it represents the convex hull that contains the eigenvalues of this matrix. Note also that, again with $U^\top V$ orthogonal, our assumption is equivalent to saying that the ratio $\frac{\mathbf{x}^* (U^\top V)\mathbf{x}}{\mathbf{x}^* \mathbf{x}}$, with $\mathbf{x}\in \mathbb{C}^s \setminus \left\{\mathbf{0}\right\}$, is contained within the right-half of the unit circle centered at the origin of the complex plane. For these as well as other results on the field of values, we recommend the book \cite{Horn_Johnson}.

\begin{remark}\label{remark_SVD}
Before moving to the statement and the proof of the eigenvalue result for the preconditioner we adopt, we would like to discuss the assumption we make. As we mentioned above, the SVD is not unique. However, under the assumption of $A_{\mathrm{RK}}$ being invertible, the field of values of the product $U^\top V$ of the matrices containing the singular vectors is uniquely defined. In fact, an invertible matrix $A_{\mathrm{RK}}$ has a unique polar decomposition $A_{\mathrm{RK}}=\widehat{U} P$, with $\widehat{U}$ unitary and $P$ Hermitian positive-definite, see \cite[Theorem\,2.17]{Hall}. Therefore, given $A_{\mathrm{RK}}= U \Sigma V^\top$ as a SVD of the matrix $A_{\mathrm{RK}}$, we clearly have $P = V \Sigma V^\top$ and $\widehat{U}=U V^\top$, with the latter implying that $U^\top V = V^\top \widehat{U}^\top V$. Finally, from the previous expression and the uniqueness of the matrix $\widehat{U}$, we can derive that the field of values of $U^\top V$ is uniquely defined.
\end{remark}

\begin{theorem}\label{optimality_preconditioner}
Let $A_{\mathrm{RK}}$ be the matrix representing the coefficients of a Runge--Kutta method. Let $A_{\mathrm{RK}}=U \Sigma V^\top$ be a SVD of the matrix $A_{\mathrm{RK}}$. Suppose that the real part of the Rayleigh quotient $\frac{\mathbf{x}^* (U^\top V)\mathbf{x}}{\mathbf{x}^* \mathbf{x}}$ is positive, for any $\mathbf{x}\in \mathbb{C}^s \setminus\left\{\mathbf{0}\right\}$. Then, the eigenvalues of the matrix $\mathcal{P}_{\mathrm{RK}}^{-1}(I_s \otimes M + \tau A_{\mathrm{RK}}\otimes K)$ all lie in the right-half of the unit circle centered at the origin of the complex plane.
\end{theorem}

\begin{proof}
Let $\lambda$ be an eigenvalue of the matrix $\mathcal{P}_{\mathrm{RK}}^{-1}(I_s \otimes M + \tau A_{\mathrm{RK}}\otimes K)$, with $\mathbf{x}$ the corresponding eigenvector. Then, we have
\begin{displaymath}
(I_s \otimes M + \tau A_{\mathrm{RK}}\otimes K) \mathbf{x} = \lambda \mathcal{P}_{\mathrm{RK}} \mathbf{x}.
\end{displaymath}
By employing the SVD of the matrix $A_{\mathrm{RK}}$ and a well known property of the Kronecker product, we can write
\begin{displaymath}
(U \otimes I_{n_x})((U^\top V) \otimes M + \tau \Sigma \otimes K)(V^\top \otimes I_{n_x}) \mathbf{x} = \lambda \mathcal{P}_{\mathrm{RK}} \mathbf{x}.
\end{displaymath}

From \eqref{prec_heat_eqaution}, by setting $\mathbf{y}=(V^\top \otimes I_{n_x}) \mathbf{x}$, the previous expression is equivalent to
\begin{displaymath}
((U^\top V) \otimes M + \tau \Sigma \otimes K) \mathbf{y} = \lambda (I_s \otimes M + \tau \Sigma \otimes K) \mathbf{y}.
\end{displaymath}
Recalling that the matrix $M$ is s.p.d., we can write $M^{\frac{1}{2}}$. Then, we have
\begin{equation}	\label{GEP}
((U^\top V) \otimes I_{n_x} + \tau \Sigma \otimes (M^{-\frac{1}{2}}KM^{-\frac{1}{2}})) \mathbf{z} = \lambda (I_s \otimes I_{n_x} + \tau \Sigma \otimes (M^{-\frac{1}{2}}KM^{-\frac{1}{2}})) \mathbf{z},
\end{equation}
where $\mathbf{z} = (I_s \otimes M^{\frac{1}{2}}) \mathbf{y}$.

Since $\Sigma$ is s.p.d. and $K$ is symmetric positive semi-definite, we have that the matrix $I_s \otimes I_{n_x} + \tau \Sigma \otimes (M^{-\frac{1}{2}}KM^{-\frac{1}{2}})$ is s.p.d., therefore invertible. Thus, we can consider the generalized Rayleigh quotient
\begin{eqnarray}
\lambda & = \dfrac{\mathbf{z}^* ((U^\top V) \otimes I_{n_x} + \tau \Sigma \otimes (M^{-\frac{1}{2}}KM^{-\frac{1}{2}})) \mathbf{z}}{\mathbf{z}^* (I_s \otimes I_{n_x} + \tau \Sigma \otimes (M^{-\frac{1}{2}}KM^{-\frac{1}{2}})) \mathbf{z}} \nonumber\\
 & = \dfrac{1}{1 + \tau \hat{\lambda}}\left(\dfrac{\mathbf{z}^* ((U^\top V) \otimes I_{n_x})\mathbf{z}}{\mathbf{z}^* \mathbf{z}} + \tau \hat{\lambda} \right) \label{lambda_heat_equation},
\end{eqnarray}
with $\hat{\lambda}= \frac{\mathbf{z}^* (\Sigma \otimes (M^{-\frac{1}{2}}KM^{-\frac{1}{2}})) \mathbf{z}}{\mathbf{z}^* \mathbf{z}}$.

Again, since the matrix $\Sigma \otimes (M^{-\frac{1}{2}}KM^{-\frac{1}{2}})$ is symmetric positive semi-definite, the value $\hat{\lambda}$ is real and non-negative. Then, since $\tau>0$, we have $1 + \tau \hat{\lambda} \geq 1$, and $\frac{\tau \hat{\lambda}}{1 + \tau \hat{\lambda}}\leq 1$.

Denoting here with $i$ the imaginary unit, under our assumption we can write
\begin{displaymath}
\dfrac{\mathbf{z}^* ((U^\top V) \otimes I_{n_x})\mathbf{z}}{\mathbf{z}^* \mathbf{z}} = a + i b,
\end{displaymath}
with $0 < a \leq 1$, and $b\in (-1,1)$ such that $a^2+b^2\leq 1$, due to $U^\top V$ being orthogonal. In fact, we have
\begin{displaymath}
\dfrac{\mathbf{z}^* ((U^\top V) \otimes I_{n_x})\mathbf{z}}{\mathbf{z}^* \mathbf{z}}=\dfrac{\hat{\mathbf{z}}^* ( I_{n_x} \otimes (U^\top V))\hat{\mathbf{z}}}{\hat{\mathbf{z}}^* \hat{\mathbf{z}}},
\end{displaymath}
with $\hat{\mathbf{z}}=P \mathbf{z}$ for a suitable permutation $P$. In particular, we have that the fields of values $\left\{\frac{\mathbf{x}^* (U^\top V)\mathbf{x}}{\mathbf{x}^* \mathbf{x}} | \mathbf{x}\in \mathbb{C}^s, \, \mathbf{x}\neq \mathbf{0} \right\}$ and $\left\{\frac{\mathbf{z}^* ((U^\top V) \otimes I_{n_x})\mathbf{z}}{\mathbf{z}^* \mathbf{z}} | \mathbf{z}\in \mathbb{C}^{sn_x}, \, \mathbf{z}\neq \mathbf{0} \right\}$ describe the same subset of the complex plane. Since $a>0$, from \eqref{lambda_heat_equation} we can say that the real part of $\lambda$ is greater than $0$. Finally, from $a^2+b^2\leq 1$ and \eqref{lambda_heat_equation} we can derive that the absolute value of $\lambda$ is less or equal than $1$, that is, $\lambda$ lies in the right-half of the unit circle centered at the origin of the complex plane. In fact, we have
\begin{align*}
| \lambda |^2 & = \lambda \lambda^* =  \dfrac{1}{(1+\tau \hat{\lambda})^2}\left[ (a+\tau \hat{\lambda})^2 + b^2 \right] =  \dfrac{1}{(1+\tau \hat{\lambda})^2}\left[ a^2 + b^2 +(\tau \hat{\lambda})^2 + 2 a \tau \hat{\lambda} \right]\\
 & \leq \dfrac{1}{(1+\tau \hat{\lambda})^2}\left[ 1 +(\tau \hat{\lambda})^2 + 2 a \tau \hat{\lambda} \right] \leq \dfrac{1}{(1+\tau \hat{\lambda})^2}\left[ 1 +(\tau \hat{\lambda})^2 + 2 \tau \hat{\lambda} \right] = 1.
\end{align*}

Since $\lambda$ is an eigenvalue of the matrix $\mathcal{P}_{\mathrm{RK}}^{-1}(I_s \otimes M + \tau A_{\mathrm{RK}}\otimes K)$, the above gives the desired result.
\end{proof}

The next lemma gives a further characterization of the eigenvalues of the preconditioned matrix $\mathcal{P}_{\mathrm{RK}}^{-1}(I_s \otimes M + \tau A_{\mathrm{RK}}\otimes K)$.
	\begin{lemma}
		Let the hypotheses of Theorem
		\ref{optimality_preconditioner} hold. If, in addition,
 $1 \in \sigma(U^\top V)$ with multiplicity $k \ge 1$, then 1 is also an eigenvalue of $\mathcal{P}_{\mathrm{RK}}^{-1}(I_s \otimes M + \tau A_{\mathrm{RK}}\otimes K)$ with (geometric) multiplicity at least $k\cdot n_x$.
	\end{lemma}
	\begin{proof}
	The orthogonal matrix $U^\top V$ can be spectrally decomposed as $U^\top V = Q \Lambda Q^*$, with $\Lambda = \text{diag}
	\left(\mu_1, \ldots, \mu_s\right)$. 
	We can then rewrite \eqref{GEP} as
\begin{displaymath}
(Q \Lambda Q^* \otimes I_{n_x} + \tau \Sigma \otimes (M^{-\frac{1}{2}}KM^{-\frac{1}{2}})) \mathbf{z} = \lambda (I_s \otimes I_{n_x} + \tau \Sigma \otimes (M^{-\frac{1}{2}}KM^{-\frac{1}{2}})) \mathbf{z}.
\end{displaymath}
	Multiplying both sides of the previous equality on the left by $Q^* \otimes I_{n_x}$ and setting $\mathbf{w} = (Q^* \otimes I_{n_x}) \mathbf{z}$  yields
\begin{displaymath}
	(\Lambda \otimes I_{n_x} + T ) \mathbf{w} = \lambda (I_s \otimes I_{n_x} + T ) \mathbf{w},
\end{displaymath}
where $T = \tau (Q^* \Sigma Q)\otimes (M^{-\frac{1}{2}}KM^{-\frac{1}{2}})$.
	Then, the previous expression gives
	\[ (I_{s \cdot n_x} + T)^{-1}(\Lambda \otimes I_{n_x} + T) \mathbf{w} = \lambda \mathbf{w}.\]
	Assuming without loss of generality that $\mu_1 = \mu_2 = \cdots = \mu_k = 1$, we can write
	\begin{eqnarray*}
		(I_{s \cdot n_x} + T)^{-1}(\Lambda \otimes I_{n_x} + T)   &= &
		(I_{s \cdot n_x} + T)^{-1}(I_{s \cdot n_x} + T + (\Lambda-I_s) \otimes I_{n_x})  \nonumber \\
		&=& I_{s \cdot n_x} +	(I_{s \cdot n_x} + T)^{-1}((\Lambda-I_s) \otimes I_{n_x}) \nonumber \\
		& =&  \begin{bmatrix} 
			\text{\large{$	I_{k\cdot n_x}$}} && \text{\large $\ast$} \\[.4em]
			\text{\large{$  0$}}   && \text{\large $\ast$}
		\end{bmatrix},
	\end{eqnarray*}
		since $\Lambda-I_s=\text{diag}(0, \ldots, 0, \mu_{k+1}-1, \ldots, \mu_s-1)$.
		 Denoting with $\mathbf e_j$ the $j$-th vector of the canonical basis, the previous equality shows that for every nonzero vector $\mathbf{t} \in \mathbb{R}^{n_x}$, $\mathbf{w} = \mathbf{e}_j \otimes \mathbf{t}, \ j = 1, \ldots, k$
		is an eigenvector of $(I_{s \cdot n_x} + T)^{-1}(\Lambda \otimes I_{n_x} + T)$ corresponding to the eigenvalue $\lambda =1$.
\end{proof}

In Figure \ref{eig_distr_heat_equation}, we report the eigenvalue distributions of the matrices $\mathcal{P}_{\mathrm{RK}}^{-1}(I_s \otimes M + \tau A_{\mathrm{RK}} \otimes K)$ and $U^\top V$, employing $\mathbf{Q}_2$ elements, for 3-stage Gauss, 4-stage Lobatto IIIC, 5-stage Radau IIA, and 9-stage Radau IIA methods, with $\tau=0.2$ and level of refinement $l=4$. Here, $l$ represents a spatial uniform grid of mesh size $h=2^{-l}$, in each dimension. Further, in green we plot the unit circle centered at the origin of the complex plane.

\begin{figure}[!htb]
\centering
\caption{Eigenvalue distributions of $\mathcal{P}_{\mathrm{RK}}^{-1}(I_s \otimes M + \tau A_{\mathrm{RK}} \otimes K)$ and of $U^\top V$, for 3-stage Gauss, 4-stage Lobatto IIIC, 5-stage Radau IIA, and 9-stage Radau IIA methods, with $\tau = 0.2$, and $l=4$. In green, we plot the unit circle centered at the origin of the complex plane.}
\label{eig_distr_heat_equation}
\begin{subfigure}[b]{0.4\textwidth}
\centering
\includegraphics[width=1.1\linewidth]{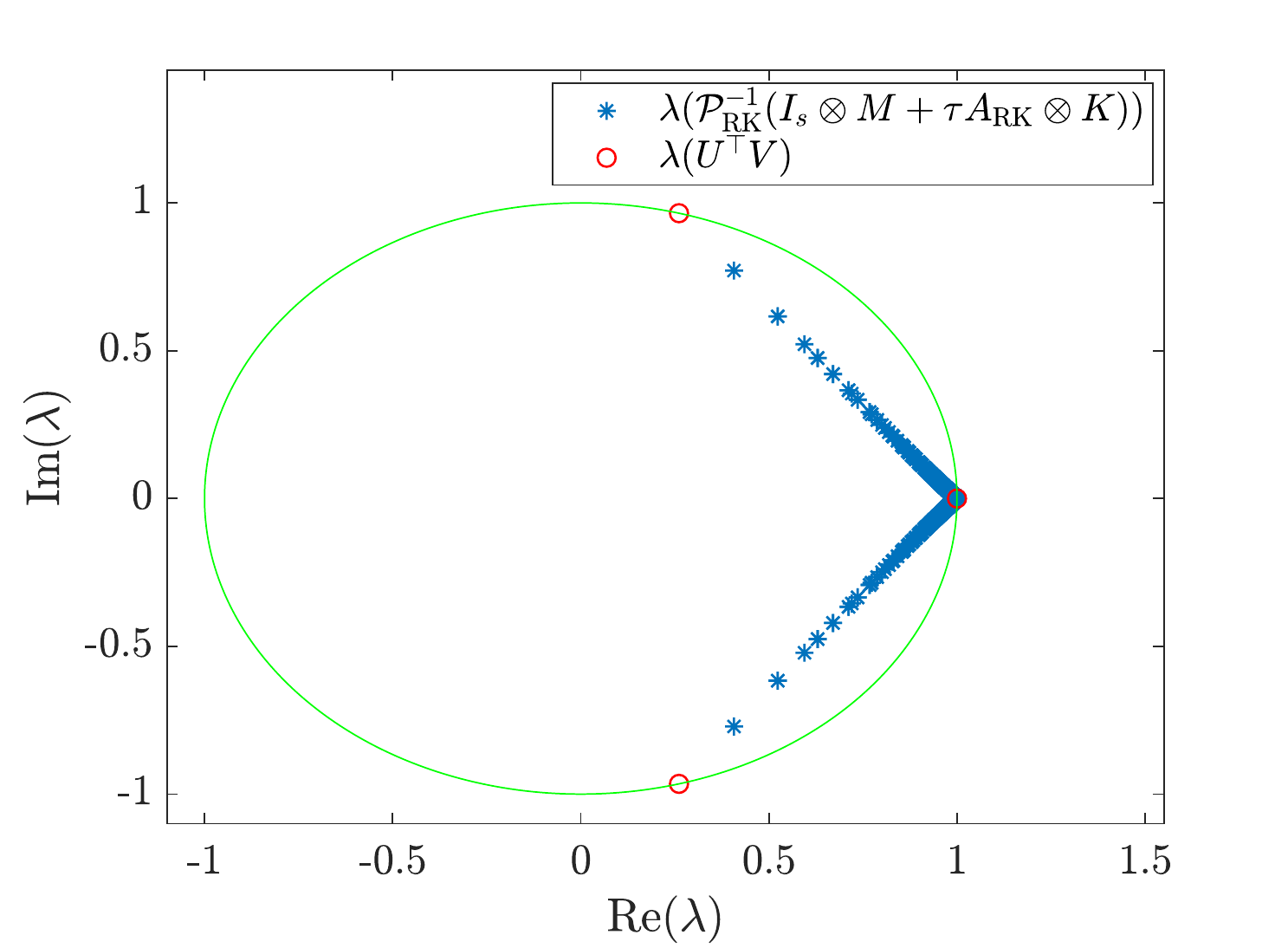}
\caption{3-stage Gauss.}
\end{subfigure}\hspace{2.5em}
\begin{subfigure}[b]{0.4\textwidth}
\centering
\includegraphics[width=1.1\linewidth]{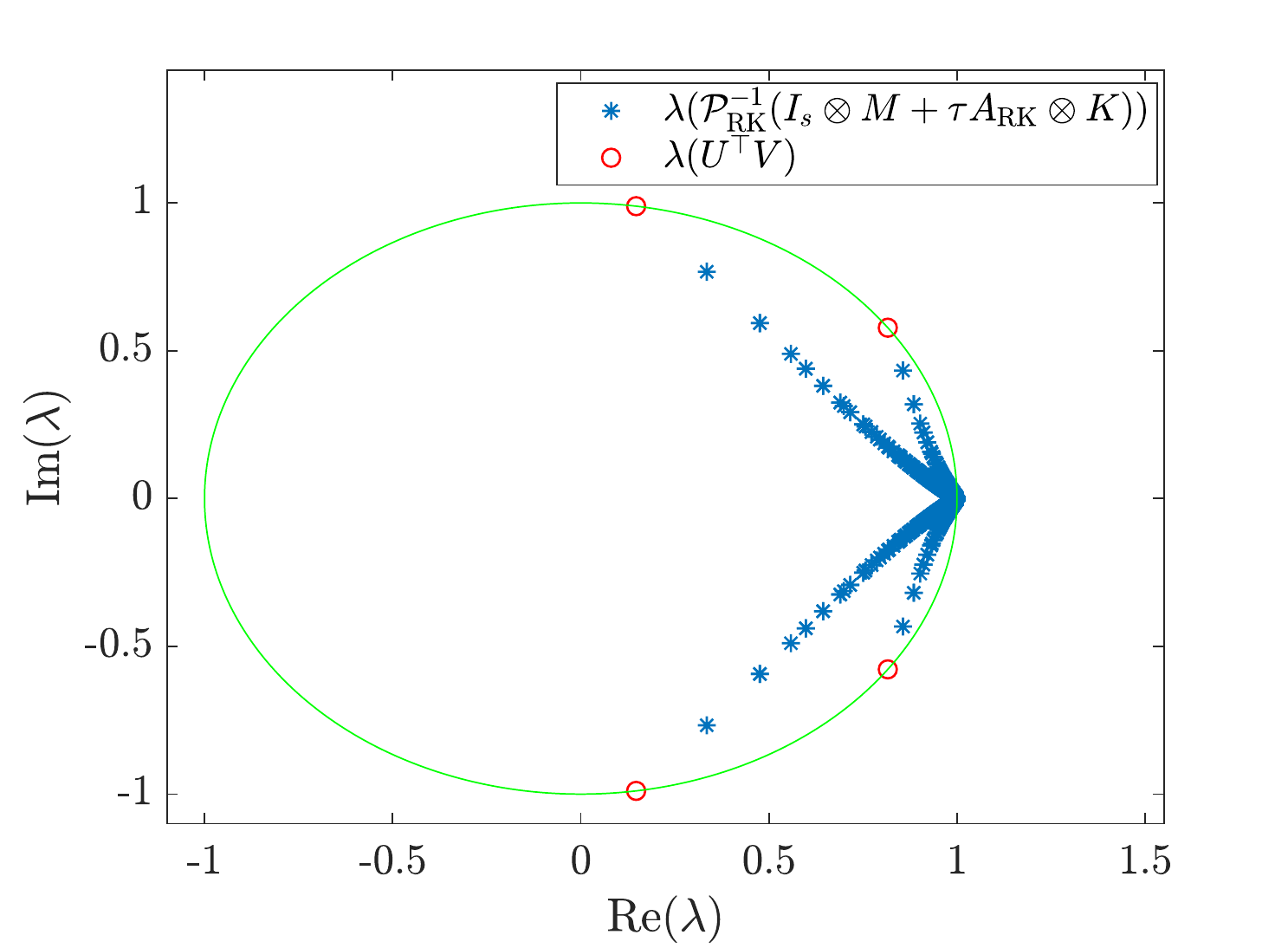}
\caption{4-stage Lobatto IIIC.}
\end{subfigure}
\begin{subfigure}[b]{0.4\textwidth}
\centering
\includegraphics[width=1.1\linewidth]{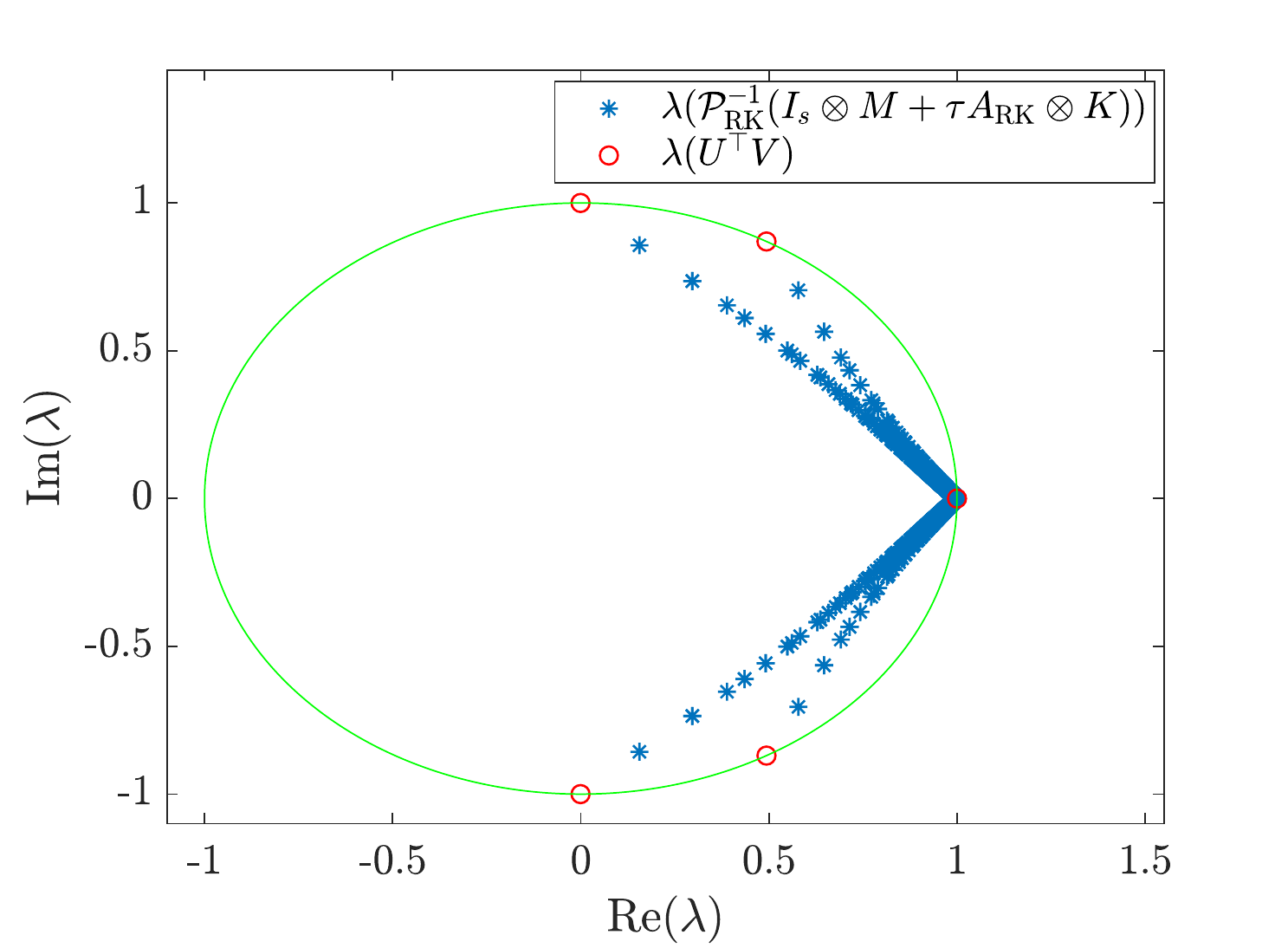}
\caption{5-stage Radau IIA.}
\end{subfigure}\hspace{2.5em}
\begin{subfigure}[b]{0.4\textwidth}
\centering
\includegraphics[width=1.1\linewidth]{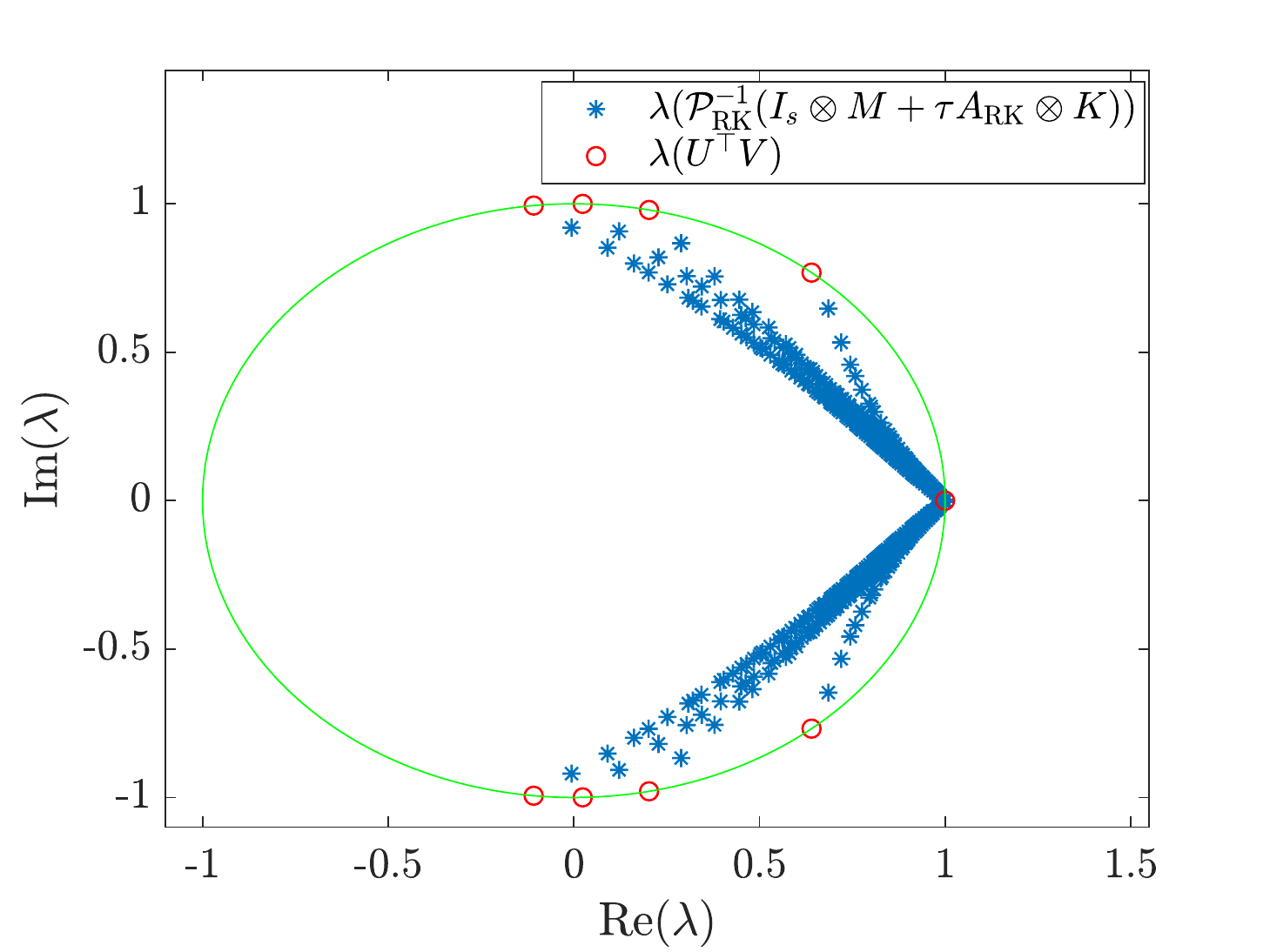}
\caption{9-stage Radau IIA.}
\end{subfigure}
\end{figure}

Interestingly, as observed in Figure \ref{eig_distr_heat_equation}, the eigenvalues of $\mathcal{P}_{\mathrm{RK}}^{-1}(I_s \otimes M + \tau A_{\mathrm{RK}} \otimes K)$ all lie in, or near, segments that join $1$ to an eigenvalue of $U^\top V$. We are not able to explain this behavior using the analysis of Theorem \ref{optimality_preconditioner}. However,  we would like to note that, in practice, the eigenvalues seem to locate away from $0$. In particular, we cannot expect $0$ to be an eigenvalue of $\mathcal{P}_{\mathrm{RK}}^{-1}(I_s \otimes M + \tau A_{\mathrm{RK}} \otimes K)$, since both $\mathcal{P}_{\mathrm{RK}}$ and $I_s \otimes M + \tau A_{\mathrm{RK}} \otimes K$ are invertible.

We would like to mention that, for all the Runge--Kutta methods we employ (excluding the 5-stage Radau IIA), the real part of the eigenvalues of the matrix $U^\top V$ is positive\footnote{For the 5-stage Radau IIA method, some of the eigenvalues of the matrix $U^\top V$ have negative, but very close to $0$, real part, as seen from Figure \ref{eig_distr_heat_equation}.}, thus the assumption of Theorem \ref{optimality_preconditioner} is not excessively restrictive. Further, following the sketch of the proof above, one can understand that we are able to derive this result about the preconditioner only because the matrix $K$ is symmetric positive semi-definite. For this reason, we cannot expect this property of the preconditioner to hold for more general problems.

\subsubsection{Stokes equations}\label{sec_3_1_2}
In this section we derive a preconditioner for the block $\widehat{\Theta}$ defined in \eqref{system_for_stages_Stokes}.

We recall that the matrix $\widehat{\Theta}$ is given by
\begin{displaymath}
\widehat{\Theta}=\left[
\begin{array}{cc}
I_s \otimes M_v + \tau A_{\mathrm{RK}}\otimes K_v & \tau A_{\mathrm{RK}}\otimes B^\top\\
\tau A_{\mathrm{RK}}\otimes B & 0
\end{array}
\right].
\end{displaymath}
In order to solve for this system, we employ as a preconditioner
\begin{displaymath}
\mathcal{P}_{\mathrm{RK}}=
\left[
\begin{array}{cc}
I_s \otimes M_v + \tau A_{\mathrm{RK}}\otimes K_v & 0\\
\tau A_{\mathrm{RK}}\otimes B & S_{\mathrm{RK}}
\end{array}
\right],
\end{displaymath}
where
\begin{equation}\label{S_RK_Stokes}
S_{\mathrm{RK}}= - \tau^2 (A_{\mathrm{RK}}\otimes B) (I_s \otimes M_v + \tau A_{\mathrm{RK}}\otimes K_v)^{-1} (A_{\mathrm{RK}}\otimes B^\top).
\end{equation}

The $(1,1)$-block $I_s \otimes M_v + \tau A_{\mathrm{RK}}\otimes K_v$ can be dealt with as for the heat equation. More specifically, we approximate it with $( U \otimes I_{n_v} ) (I_s \otimes M_v + \tau \Sigma \otimes K_v) ( V^\top \otimes I_{n_v} )$.

In order to find a suitable approximation of the Schur complement $S_\mathrm{RK}$, we observe that
\begin{align*}
\vspace{0.5ex}
A_{\mathrm{RK}}\otimes B & = (A_{\mathrm{RK}}\otimes I_{n_p}) 
(I_s \otimes B),\\
A_{\mathrm{RK}}\otimes B^\top & =
(I_s \otimes B^\top)
(A_{\mathrm{RK}}\otimes I_{n_p}) .
\end{align*}
It is clear that if we find a suitable approximation of the following matrix:
\begin{displaymath}
\widetilde{S}_{\mathrm{int}} \approx S_{\mathrm{int}}=
(I_s \otimes B)
(I_s \otimes M_v + \tau A_{\mathrm{RK}}\otimes K_v)^{-1}
(I_s \otimes B^\top),
\end{displaymath}
then a suitable approximation of the Schur complement $S_{\mathrm{RK}}$ is given by
\begin{equation}\label{widetilde_S_RK}
\widetilde{S}_{\mathrm{RK}}:= - \tau^2 (A_{\mathrm{RK}}\otimes I_{n_p})\widetilde{S}_{\mathrm{int}} (A_{\mathrm{RK}}\otimes I_{n_p}).
\end{equation}
Note that the matrix $(A_{\mathrm{RK}}\otimes I_{n_p})$ can be easily inverted in parallel by making use of the SVD of the matrix $A_{\mathrm{RK}}=U \Sigma V^\top$.

In order to derive a suitable approximation of the matrix $S_{\mathrm{int}}$, we employ the block-commutator argument derived in \cite{Leveque_Pearson}. We would like to mention that a similar approach has been derived independently and employed for another parallel-in-time solver for the incompressible Navier--Stokes equations by the authors in \cite{Danieli_Southworth_Wathen}. The approximation we employ is given by
\begin{displaymath}
\widetilde{S}_{\mathrm{int}} := 
(I_s \otimes K_p)
(I_s \otimes M_p + \tau A_{\mathrm{RK}}\otimes K_p)^{-1}
(I_s \otimes M_p).
\end{displaymath}
Then, our approximation of the Schur complement $S_{\mathrm{RK}}$ is given by \eqref{widetilde_S_RK} with this choice of $\widetilde{S}_{\mathrm{int}}$. For details on the derivation of the approximation, we refer to \cite{Leveque_Pearson}.

In Figure \ref{eig_distr_Schur_compl_Stokes}, we report the eigenvalues of the matrix $\widetilde{S}_{\mathrm{RK}}^{-1}S_{\mathrm{RK}}$, for 3-stage Gauss, 3-stage Lobatto IIIC, and 3-stage Radau IIA methods, with $\tau=0.2$, and level of refinement $l=5$. Here, $l$ represents a (spatial) uniform grid of mesh size $h=2^{1-l}$ for $\mathbf{Q}_1$ basis functions, and $h=2^{-l}$ for $\mathbf{Q}_2$ elements, in each dimension. Since for the problem we are considering the matrix $K_p$ is not invertible, we derive an invertible approximation $\widetilde{S}_{\mathrm{RK}}$ by ``pinning'' the value of one of the nodes of the matrix $K_p$, for each $K_p$ within the definition of $\widetilde{S}_{\mathrm{int}}$.

\begin{figure}[!htb]
\centering
\caption{Block-commutator approximation for the Stokes equations. Eigenvalues of $\widetilde{S}_{\mathrm{RK}}^{-1}S_{\mathrm{RK}}$, for 3-stage Gauss, 3-stage Lobatto IIIC, and 3-stage Radau IIA methods, with $\tau = 0.2$, and $l=5$.}
\label{eig_distr_Schur_compl_Stokes}
\begin{subfigure}[b]{0.4\textwidth}
\centering
\includegraphics[width=1.1\linewidth]{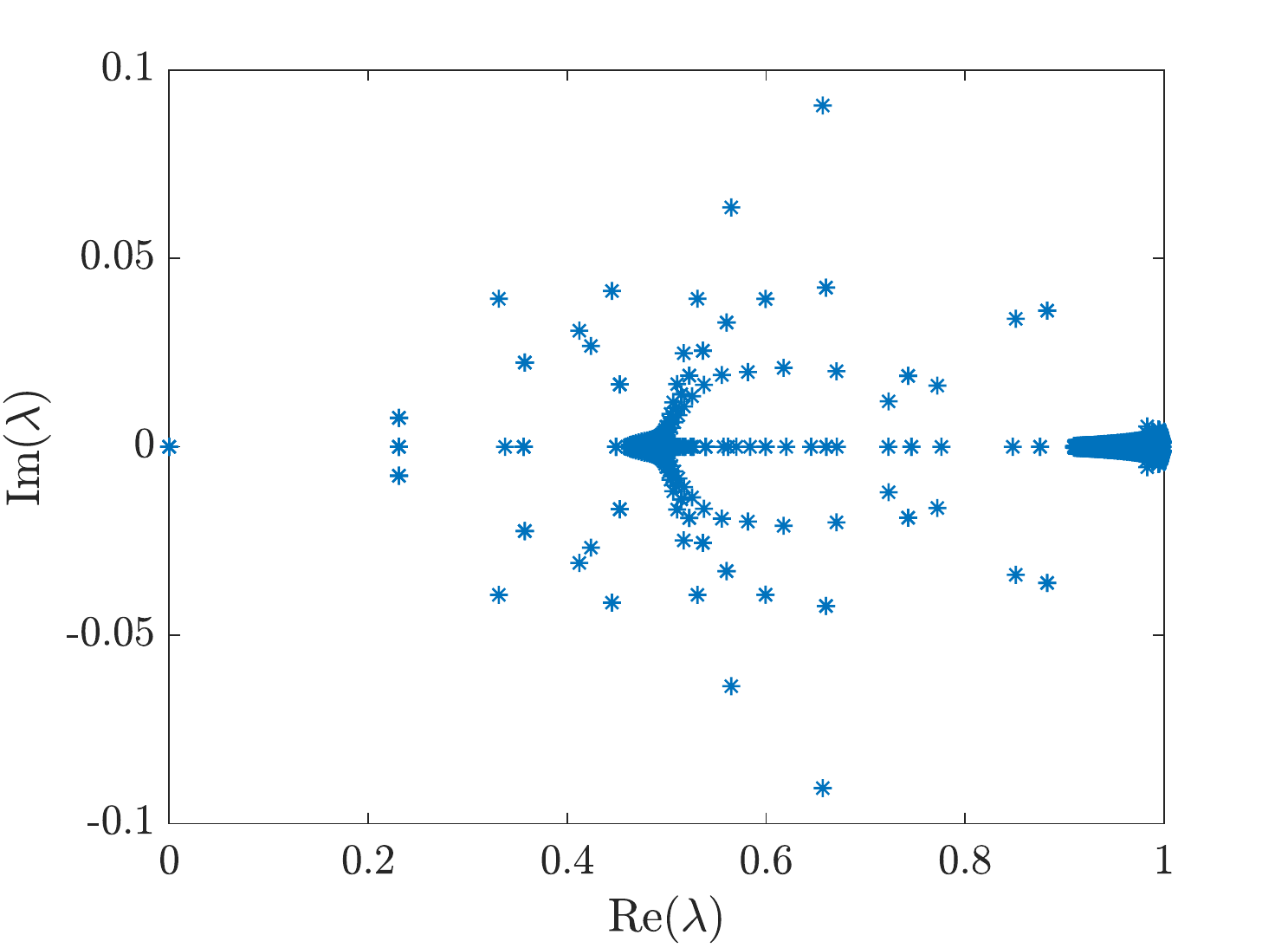}
\caption{3-stage Gauss.}
\end{subfigure}\hspace{2.5em}
\begin{subfigure}[b]{0.4\textwidth}
\centering
\includegraphics[width=1.1\linewidth]{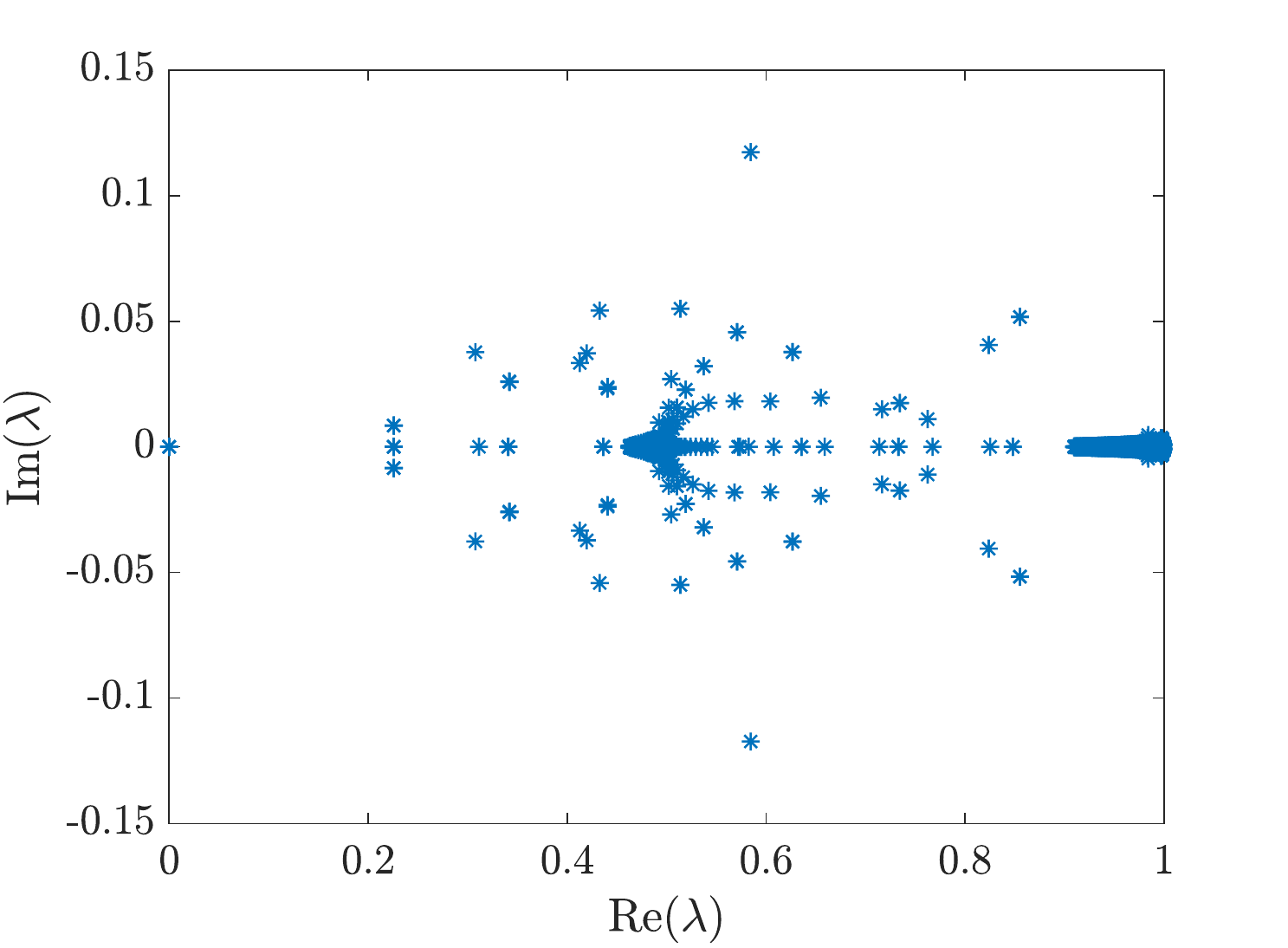}
\caption{3-stage Lobatto IIIC.}
\end{subfigure}
\begin{subfigure}[b]{0.4\textwidth}
\centering
\includegraphics[width=1.1\linewidth]{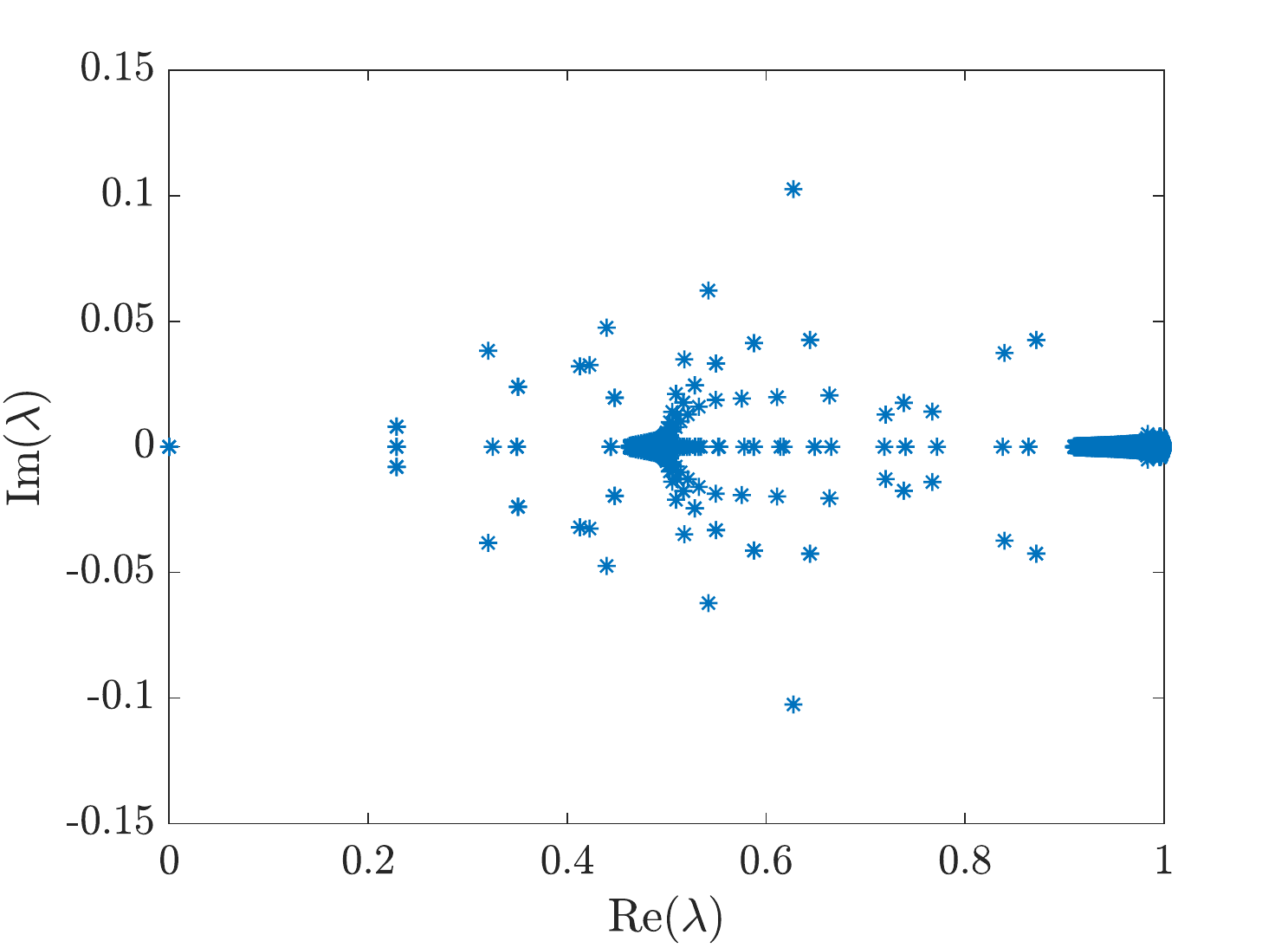}
\caption{3-stage Radau IIA.}
\end{subfigure}
\end{figure}

Finally, we can present our approximation $\widetilde{\mathcal{P}}_{\mathrm{RK}}$ of the preconditioner $\mathcal{P}_{\mathrm{RK}}$ for the matrix $\widehat{\Theta}$. By employing a well known property of the Kronecker product, the approximation of the preconditioner $\mathcal{P}_{\mathrm{RK}}$ is given by the following:
\begin{equation}\label{prec_Stokes_equation}
\widetilde{\mathcal{P}}_{\mathrm{RK}} = 
\left[
\begin{array}{cc}
U \otimes I_{n_v} & 0\\
0 & U \otimes I_{n_p}
\end{array}
\right]
\widetilde{\mathcal{P}}_{\mathrm{int}}
\left[
\begin{array}{cc}
V^\top \otimes I_{n_v} & 0\\
0 & V^\top \otimes I_{n_p}
\end{array}
\right],
\end{equation}
with
\begin{displaymath}
\widetilde{\mathcal{P}}_{\mathrm{int}}=
\left[
\begin{array}{cc}
I_s \otimes M + \tau \Sigma \otimes K & 0\\
\tau \Sigma \otimes B & -\tau^2 \left( (\Sigma V^\top) \otimes I_{n_p}\right) \widetilde{S}_\mathrm{int} \left( (U\Sigma) \otimes I_{n_p}\right)
\end{array}
\right].
\end{displaymath}

Before completing this section, we would like to mention that the matrix $\widehat{\Theta}$ is not invertible. For this reason, the block $\Theta = I_{n_t} \otimes \widehat{\Theta}$ is not invertible, and one cannot employ the preconditioner defined in \eqref{preconditioner_all_at_once} as it is. We thus rather employ as $(2,2)$-block the matrix $\widetilde{\Theta} = I_{n_t} \otimes \widetilde{\Theta}$, with $\widetilde{\Theta}$ given by the following invertible perturbation of $\widehat{\Theta}$:
\begin{displaymath}
\widetilde{\Theta}=\left[
\begin{array}{cc}
I_s \otimes M_v + \tau A_{\mathrm{RK}}\otimes K_v & \tau A_{\mathrm{RK}}\otimes B^\top\\
\tau A_{\mathrm{RK}}\otimes B & - \tau^2 \gamma (A_{\mathrm{RK}}^2 \otimes M_p)
\end{array}
\right],
\end{displaymath}
where $0<\gamma\ll 1$. The parameter $\gamma$ has to be chosen such that the perturbed matrix $\widetilde{\Theta}$ is invertible, but ``close'' enough to the original $\widehat{\Theta}$. In our tests, we chose $\gamma=10^{-4}$. Note that we chose this perturbation of the $(2,2)$-block for the expression of the Schur complement $S_{\mathrm{RK}}$ given in \eqref{S_RK_Stokes}. In fact, it is easy to see that
\begin{displaymath}
- \tau^2 \gamma (A_{\mathrm{RK}}^2 \otimes M_p) = - \tau^2 (A_{\mathrm{RK}} \otimes I_{n_p}) (\gamma I_s \otimes M_p) (A_{\mathrm{RK}} \otimes I_{n_p}),
\end{displaymath}
implying that we perturb $S_{\mathrm{int}}$ only by the matrix $\gamma I_s \otimes M_p$.

\section{Numerical results}\label{sec_4}
We now provide numerical evidence of the effectiveness of our preconditioning strategies. In all our tests, $d=2$ (that is, $\mathbf{x}=(x_1,x_2)$), and $\Omega=(-1,1)^2$. We employ $\mathbf{Q}_1$ and $\mathbf{Q}_2$ finite element basis functions in the spatial dimensions for the heat equation, while employing inf--sup stable Taylor--Hood $\mathbf{Q}_2$--$\mathbf{Q}_1$ elements for the Stokes equations. In all our tests, the level of refinement $l$ represents a (spatial) uniform grid of mesh size $h=2^{1-l}$ for $\mathbf{Q}_1$ basis functions, and $h=2^{-l}$ for $\mathbf{Q}_2$ elements, in each dimension. Regarding the time grid, denoting with $q_{\mathrm{FE}}$ the order of the finite element approximations and with $q_{\mathrm{RK}}$ the order of the Runge--Kutta method we employ, we choose as the number of intervals in time $n_t$ the closest integer such that $\tau \leq h^{q_{\mathrm{FE}}/q_{\mathrm{RK}}}$, where $\tau=t_f/n_t$ is the time-step.

In all our tests, we apply 20 steps of Chebyshev semi-iteration \cite{GolubVargaI,GolubVargaII,Wathen_Rees} with Jacobi splitting to approximately solve systems involving mass matrices, while employing 2 V-cycles (with 2 symmetric Gauss--Seidel iterations for pre-/post-smoothing) of the \texttt{HSL\_MI20} solver \cite{HSL_MI20} for other matrices.

\subsection{Sequential results}
Before showing the parallel efficiency of our solver, we present numerical results in a sequential framework, showing the robustness of our approximations of the main blocks of the system considered. All sequential tests are run on MATLAB R2018b, using a 1.70GHz Intel quad-core i5 processor and 8 GB RAM on an Ubuntu 18.04.1 LTS operating system. All CPU times below are reported in seconds.

\subsubsection{Heat equation: sequential time-stepping}\label{heat_equation:sequential_time_stepping}
In this section, we compare our SVD-based preconditioner with the block-diagonal preconditioner $\mathcal{P}_{\mathrm{MNS}}$ derived in \cite{Mardal_Nilssen_Staff}. We refer the reader to Section \ref{sec_3_1_1} for this choice of preconditioner. We evaluate the numerical solution in a classical way, solving sequentially for the time step $n$, that is, we solve for the system of the stages given in \eqref{RK_stages_heat_equation} and then update the solution at time $t_{n+1}$ with \eqref{RK_discretization_for_v_heat_equation}. We compare the number of iterations required for convergence up to a tolerance $10^{-8}$ on the relative residual for both methods. We employ GMRES as the Krylov solver, with a restart every 10 iterations.

We take $t_f=2$, and solve the problem \eqref{heat_equation} with solution given by
\begin{displaymath}
v (\mathbf{x},t) = e^{t_f-t} \cos\left(\frac{\pi x_1}{2}\right) \cos\left(\frac{\pi x_2}{2}\right)+1,
\end{displaymath}
with initial and boundary conditions obtained from this $v$. In Tables \ref{sequential_heat_equation_table_Gauss}--\ref{sequential_heat_equation_table_Radau}, we report the average number of GMRES iterations and the total CPU times for solving all the linear systems together with the numerical errors $v_{\text{error}}$ of the solutions\footnote{We report only one value of $v_{\text{error}}$ per finite element discretization, as the numerical errors obtained when using both preconditioners coincide at least up to the second significant digit, excluding the pollution of the linear residual on the solution. The value of $v_{\text{error}}$ that we report is the one obtained with our preconditioner $\mathcal{P}_{\mathrm{RK}}$.}, for different Runge--Kutta methods. We evaluate the numerical error $v_{\text{error}}$ in the scaled vector $\ell^\infty$-norm, defined as
\begin{displaymath}
v_{\text{error}} = \max_{n} \left\{ \dfrac{| v_{\mathtt{j},n} - v_{\mathtt{j},n}^{\mathrm{\: sol}} |}{|v_{\mathtt{j},n}^{\mathrm{\: sol}}|}, \; \mathrm{with} \; \mathtt{j} = \argmax_j | v_{j,n} - v_{j,n}^{\mathrm{\: sol}} | \right\},
\end{displaymath}
where $v_{j,n}$ and $v_{j,n}^{\mathrm{\: sol}}$ are the entries of the computed solution $\mathbf{v}$ and the (discretized) exact solution for $v$ at time $t=t_n$. Further, in Table \ref{sequential_heat_equation_table_Radau} we report the dimension of the system solved for Radau IIA methods when employing $\mathbf{Q}_1$ and $\mathbf{Q}_2$ finite elements, respectively. We would like to note that those values coincide with the dimension of the system arising when employing the corresponding level of refinement of an $s$-stage Gauss or Lobatto IIIC method.

\begin{table}[!ht]
\caption{Sequential solve of the heat equation: average GMRES iterations, total CPU times, and resulting relative errors in $v$ for Gauss methods, employing $\mathbf{Q}_1$ and $\mathbf{Q}_2$ finite elements.}\label{sequential_heat_equation_table_Gauss}
\begin{center}
\begin{footnotesize}
{\begin{tabular}{|c|c||c|c|c|c|c||c|c|c|c|c|}
\hline
\multicolumn{1}{|c}{} & \multicolumn{1}{|c||}{} & \multicolumn{5}{c||}{$\mathbf{Q}_1$} & \multicolumn{5}{c|}{$\mathbf{Q}_2$} \\
\cline{3-12}
\multicolumn{1}{|c}{} & \multicolumn{1}{|c||}{} & \multicolumn{2}{c|}{$\mathcal{P}_{\mathrm{MNS}}$} & \multicolumn{2}{c|}{$\mathcal{P}_{\mathrm{RK}}$} & \multicolumn{1}{c||}{} & \multicolumn{2}{c|}{$\mathcal{P}_{\mathrm{MNS}}$} & \multicolumn{2}{c|}{$\mathcal{P}_{\mathrm{RK}}$} & \multicolumn{1}{c|}{}\\
\cline{3-6} \cline{8-11}
$s$ & $l$ & $\texttt{it}$ & CPU & $\texttt{it}$ & CPU & $v_{\text{error}}$ & $\texttt{it}$ & CPU & $\texttt{it}$ & CPU & $v_{\text{error}}$\\
\hline
\hline
\multirow{5}{*}{2} & $3$ & 10 & 0.05 & 8 & 0.05 & 6.35e-03  & 10 & 0.1 & 8 & 0.1 & 1.33e-04 \\
 & $4$ & 10 & 0.1 & 8 & 0.09 & 1.69e-03 & 10 & 0.4 & 9 & 0.4 & 1.62e-05 \\
 & $5$ & 10 & 0.3 & 8 & 0.2 & 4.55e-04 & 10 & 2.2 & 10 & 2.2 & 2.39e-06 \\
 & $6$ & 10 & 1.3 & 9 & 1.2 & 1.14e-04 & 10 & 16 & 12 & 19 & 2.91e-07 \\
 & $7$ & 10 & 6.7 & 10 & 6.8 & 2.95e-05 & 10 & 116 & 13 & 167 & 3.45e-08 \\
\hline
\multirow{5}{*}{3} & $3$ & 18 & 0.2 & 12 & 0.1 & 5.45e-03 & 18 & 0.2 & 11 & 0.2 & 5.10e-05 \\
 & $4$ & 17 & 0.2 & 11 & 0.1 & 1.53e-03 & 18 & 0.6 & 13 & 0.5 & 3.43e-06 \\
 & $5$ & 16 & 0.5 & 12 & 0.4 & 4.17e-04 & 18 & 3.1 & 15 & 2.5 & 2.36e-07 \\
 & $6$ & 16 & 1.8 & 12 & 1.6 & 1.07e-04 & 17 & 20 & 17 & 18 & 1.66e-08 \\
 & $7$ & 16 & 8.0 & 13 & 7.0 & 2.71e-05 & 17 & 107 & 18 & 122 & 1.78e-09 \\
\hline
\end{tabular}}
\end{footnotesize}
\end{center}
\end{table}

\begin{table}[!ht]
\caption{Sequential solve of the heat equation: average GMRES iterations, total CPU times, and resulting relative errors in $v$ for Lobatto IIIC methods, employing $\mathbf{Q}_1$ and $\mathbf{Q}_2$ finite elements.}\label{sequential_heat_equation_table_Lobatto}
\begin{center}
\begin{footnotesize}
{\begin{tabular}{|c|c||c|c|c|c|c||c|c|c|c|c|}
\hline
\multicolumn{1}{|c}{} & \multicolumn{1}{|c||}{} & \multicolumn{5}{c||}{$\mathbf{Q}_1$} & \multicolumn{5}{c|}{$\mathbf{Q}_2$} \\
\cline{3-12}
\multicolumn{1}{|c}{} & \multicolumn{1}{|c||}{} & \multicolumn{2}{c|}{$\mathcal{P}_{\mathrm{MNS}}$} & \multicolumn{2}{c|}{$\mathcal{P}_{\mathrm{RK}}$} & \multicolumn{1}{c||}{} & \multicolumn{2}{c|}{$\mathcal{P}_{\mathrm{MNS}}$} & \multicolumn{2}{c|}{$\mathcal{P}_{\mathrm{RK}}$} & \multicolumn{1}{c|}{}\\
\cline{3-6} \cline{8-11}
$s$ & $l$ & $\texttt{it}$ & CPU & $\texttt{it}$ & CPU & $v_{\text{error}}$ & $\texttt{it}$ & CPU & $\texttt{it}$ & CPU & $v_{\text{error}}$\\
\hline
\hline
\multirow{5}{*}{2} & $3$ & 12 & 0.1 & 7 & 0.09 & 1.36e-02 & 13 & 0.5 & 9 & 0.3 & 2.55e-03 \\
 & $4$ & 12 & 0.3 & 8 & 0.2 & 4.10e-03 & 14 & 2.8 & 10 & 1.9 & 3.66e-04 \\
 & $5$ & 12 & 1.4 & 9 & 0.9 & 1.14e-03 & 13 & 23 & 12 & 22 & 5.06e-05 \\
 & $6$ & 12 & 8.7 & 10 & 7.1 & 3.01e-04 & 13 & 276 & 12 & 263 & 6.45e-06 \\
 & $7$ & 12 & 66 & 12 & 68 & 7.77e-05 & 12 & 3226 & 12 & 3377 & 8.17e-07 \\
\hline
\multirow{5}{*}{3} & $3$ & 26 & 0.2 & 8 & 0.07 & 5.42e-03 & 25 & 0.5 & 10 & 0.2 & 1.15e-04 \\
 & $4$ & 25 & 0.3 & 9 & 0.1 & 1.48e-03 & 26 & 1.6 & 11 & 0.7 & 1.75e-05 \\
 & $5$ & 25 & 0.9 & 10 & 0.4 & 3.79e-04 & 29 & 9.5 & 12 & 4.1 & 2.95e-06 \\
 & $6$ & 25 & 4.7 & 11 & 2.2 & 9.82e-05 & 28 & 64 & 15 & 37 & 3.89e-07 \\
 & $7$ & 26 & 25 & 12 & 13 & 2.43e-05 & 28 & 490 & 17 & 330 & 4.81e-08 \\
\hline
\multirow{5}{*}{4} & $3$ & 40 & 0.4 & 12 & 0.1 & 5.75e-03 & 41 & 0.6 & 13 & 0.2 & 2.37e-05 \\
 & $4$ & 37 & 0.4 & 11 & 0.1 & 1.55e-03 & 42 & 2.0 & 15 & 0.7 & 1.66e-06 \\
 & $5$ & 37 & 1.4 & 13 & 0.5 & 4.18e-04 & 44 & 9.2 & 16 & 3.5 & 1.61e-07 \\
 & $6$ & 38 & 5.4 & 15 & 2.4 & 1.07e-04 & 44 & 60 & 18 & 27 & 1.28e-08 \\
 & $7$ & 39 & 25 & 16 & 11 & 2.73e-05 & 43 & 368 & 18 & 171 & 2.32e-09 \\
\hline
\multirow{5}{*}{5} & $3$ & 56 & 0.5 & 15 & 0.2 & 5.91e-03 & 55 & 1.0 & 15 & 0.3 & 1.94e-05 \\
 & $4$ & 52 & 0.8 & 15 & 0.2 & 1.55e-03 & 56 & 2.8 & 17 & 0.9 & 1.19e-06 \\
 & $5$ & 51 & 1.8 & 14 & 0.5 & 4.07e-04 & 59 & 12 & 18 & 3.7 & 7.39e-08 \\
 & $6$ & 53 & 6.9 & 16 & 2.3 & 1.07e-04 & 57 & 67 & 21 & 27 & 3.95e-09 \\
 & $7$ & 54 & 34 & 17 & 11 & 2.71e-05 & 57 & 389 & 22 & 162 & 7.53e-09 \\
\hline
\end{tabular}}
\end{footnotesize}
\end{center}
\end{table}

\begin{table}[!ht]
\caption{Sequential solve of the heat equation: degrees of freedom (DoF), average GMRES iterations, total CPU times, and resulting relative errors in $v$ for Radau IIA methods, employing $\mathbf{Q}_1$ and $\mathbf{Q}_2$ finite elements.}\label{sequential_heat_equation_table_Radau}
\begin{center}
\begin{footnotesize}
{\begin{tabular}{|c|c||c||c|c|c|c|c||c||c|c|c|c|c|}
\hline
\multicolumn{1}{|c|}{} & \multicolumn{1}{c||}{} & \multicolumn{6}{c||}{$\mathbf{Q}_1$} & \multicolumn{6}{c|}{$\mathbf{Q}_2$} \\
\cline{3-14}
\multicolumn{1}{|c|}{} & \multicolumn{1}{c||}{} & \multicolumn{1}{c||}{} & \multicolumn{2}{c|}{$\mathcal{P}_{\mathrm{MNS}}$} & \multicolumn{2}{c|}{$\mathcal{P}_{\mathrm{RK}}$} & \multicolumn{1}{c||}{} & \multicolumn{1}{c||}{} & \multicolumn{2}{c|}{$\mathcal{P}_{\mathrm{MNS}}$} & \multicolumn{2}{c|}{$\mathcal{P}_{\mathrm{RK}}$} & \multicolumn{1}{c|}{}\\
\cline{4-7} \cline{10-13}
$s$ & $l$ & DoF & $\texttt{it}$ & CPU & $\texttt{it}$ & CPU & $v_{\text{error}}$ & DoF & $\texttt{it}$ & CPU & $\texttt{it}$ & CPU & $v_{\text{error}}$ \\
\hline
\hline
\multirow{5}{*}{2} & $3$ & 98 & 11 & 0.09 & 8 & 0.08 & 5.48e-03 & 450 & 12 & 0.2 & 8 & 0.2 & 2.10e-04 \\
 & $4$ & 450 & 10 & 0.1 & 8 & 0.08 & 1.39e-03 & 1922 & 12 & 0.8 & 10 & 0.6 & 2.86e-05 \\
 & $5$ & 1922 & 12 & 0.5 & 9 & 0.4 & 3.67e-04 & 7938 & 12 & 5.4 & 11 & 5.1 & 3.76e-06 \\
 & $6$ & 7938 & 12 & 2.9 & 10 & 2.3 & 9.44e-05 & 32,258 & 12 & 45 & 14 & 52 & 4.82e-07 \\
 & $7$ & 32,258 & 12 & 17 & 11 & 16 & 2.34e-05 & 130,050 & 12 & 400 & 14 & 496 & 6.11e-08 \\
\hline
\multirow{5}{*}{3} & $3$ & 147 & 21 & 0.1 & 10 & 0.07 & 5.71e-03 & 675 & 20 & 0.3 & 11 & 0.1 & 2.62e-05 \\
 & $4$ & 675 & 19 & 0.2 & 10 & 0.1 & 1.60e-03 & 2883 & 23 & 0.9 & 12 & 0.5 & 2.61e-06 \\
 & $5$ & 2883 & 19 & 0.6 & 11 & 0.4 & 4.17e-04 & 11,907 & 22 & 4.7 & 15 & 3.3 & 2.31e-07 \\
 & $6$ & 11,907 & 20 & 2.6 & 12 & 1.6 & 1.08e-04 & 48,387 & 23 & 32 & 17 & 24 & 2.96e-08 \\
 & $7$ & 48,387 & 20 & 14 & 13 & 9.3 & 2.74e-05 & 195,075 & 22 & 215 & 18 & 183 & 2.85e-09 \\
\hline
\multirow{5}{*}{4} & $3$ & 196 & 30 & 0.2 & 12 & 0.08 & 5.59e-03 & 900 & 28 & 0.4 & 15 & 0.2 & 1.94e-05 \\
 & $4$ & 900 & 29 & 0.4 & 12 & 0.2 & 1.55e-03 & 3844 & 30 & 1.1 & 16 & 0.6 & 1.17e-06 \\
 & $5$ & 3844 & 27 & 1.0 & 15 & 0.5 & 4.18e-04 & 15,876 & 30 & 5.3 & 17 & 3.2 & 7.14e-08 \\
 & $6$ & 15,876 & 27 & 3.6 & 16 & 2.1 & 1.07e-04 & 64,516 & 30 & 31 & 18 & 19 & 3.43e-09 \\
 & $7$ & 64,516 & 28 & 17 & 17 & 10 & 2.70e-05 & 260,100 & 30 & 203 & 19 & 126 & 1.20e-09 \\
\hline
\multirow{5}{*}{5} & $3$ & 245 & 40 & 0.4 & 16 & 0.1 & 5.91e-03 & 1125 & 38 & 0.7 & 16 & 0.3 & 1.91e-05 \\
 & $4$ & 1125 & 38 & 0.6 & 16 & 0.2 & 1.55e-03 & 4805 & 40 & 1.6 & 16 & 0.7 & 1.20e-06 \\
 & $5$ & 4805 & 36 & 1.2 & 15 & 0.5 & 4.07e-04 & 19,845 & 42 & 8.3 & 19 & 3.9 & 7.39e-08 \\
 & $6$ & 19,845 & 36 & 4.7 & 17 & 2.2 & 1.07e-04 & 80,645 & 40 & 41 & 21 & 23 & 4.05e-09 \\
 & $7$ & 80,645 & 37 & 25 & 18 & 12 & 2.71e-05 & 325,125 & 40 & 235 & 22 & 130 & 7.42e-09 \\
\hline
\end{tabular}}
\end{footnotesize}
\end{center}
\end{table}

Tables \ref{sequential_heat_equation_table_Gauss}--\ref{sequential_heat_equation_table_Radau} show a very mild dependence of our SVD-based preconditioner with respect to the mesh size $h$ and the number of stages $s$. Nonetheless, the preconditioner we propose is able to reach convergence in less than 25 iterations. Although the  block-diagonal preconditioner $\mathcal{P}_{\mathrm{MNS}}$ is optimal with respect to the discretization parameters, the number of iterations required to reach the prescribed tolerance is not robust with respect to the number of stages $s$. For instance, in order to reach convergence the solver needs almost 60 iterations for the 5-stage Lobatto IIIC method. On the other hand, our solver does not suffer (drastically) from this dependence. In fact, when employing a high number of stages, our new solver can be between 2 and 3 times faster. Finally, we would like to note that the error behaves as predicted, that is, the method behaves as a second- and third-order method with $\mathbf{Q}_1$ and $\mathbf{Q}_2$ finite elements, respectively. In particular, we do not experience any order reduction for the numerical tests carried out in this work, aside the pollution of the linear residual on the numerical solution.

\subsubsection{Heat equation: sequential all-at-once solve}
In this section, we present the numerical results for the (sequential) solve of the all-at-once Runge--Kutta discretization of the heat equation given in \eqref{system_all_at_once_RK_heat_equation}. The problem we consider here is the same as the one in Section \ref{heat_equation:sequential_time_stepping}. We employ as a preconditioner the matrix $\mathcal{P}$ given in \eqref{preconditioner_all_at_once}, with the matrix $\widehat{S}$ defined in \eqref{widehat_S} approximated with a block-forward substitution. Within this process, the matrix $\widehat{X}$ defined in \eqref{widehat_X} is applied inexactly through an approximate inversion of the block for the stages. More specifically, given the results of the previous section, we approximately invert each block for the stages with 5 GMRES iterations preconditioned with the matrix $\mathcal{P}_{\mathrm{RK}}$ defined in \eqref{prec_heat_eqaution}. This process is also employed for approximating the $(2,2)$-block $\Theta$. Since the inner iteration is based on a fixed number of GMRES iterations, we have to employ the flexible version of GMRES derived in \cite{Saad} as an outer solver. As above, we restart FGMRES after every 10 iterations. Our implementation is based on the flexible GMRES routine in the TT-Toolbox \cite{TT_Toolbox}. The solver is run until a relative tolerance of $10^{-8}$ is achieved.

In Tables \ref{sequential_allatonce_heat_equation_table_Gauss}--\ref{sequential_allatonce_heat_equation_table_Radau}, we report the number of FGMRES iterations and the CPU times, for different Runge--Kutta methods. Further, we report the numerical error $v_{\text{error}}$ in the scaled vector $\ell^\infty$-norm, defined as in the previous section, together with the dimension of the system solved for each method.

\begin{table}[!ht]
\caption{All-at-once solve of the heat equation: degrees of freedom (DoF), FGMRES iterations, CPU times, and resulting relative errors in $v$ for Gauss methods, employing $\mathbf{Q}_1$ and $\mathbf{Q}_2$ finite elements.}\label{sequential_allatonce_heat_equation_table_Gauss}
\begin{center}
\begin{footnotesize}
{\begin{tabular}{|c|c||c||c|c|c||c||c|c|c|}
\hline
\multicolumn{1}{|c}{} & \multicolumn{1}{|c||}{} & \multicolumn{4}{c||}{$\mathbf{Q}_1$} & \multicolumn{4}{c|}{$\mathbf{Q}_2$} \\
\cline{3-10}
$s$ & $l$ & DoF & $\texttt{it}$ & CPU & $v_{\text{error}}$ & DoF & $\texttt{it}$ & CPU & $v_{\text{error}}$ \\
\hline
\hline
\multirow{4}{*}{2} & $3$ & 637 & 6 & 0.3 & 6.06e-03 & 4275 & 6 & 0.8 & 1.08e-04 \\
 & $4$ & 4275 & 6 & 0.7 & 1.45e-03 & 29,791 & 8 & 4.0 & 1.44e-05 \\
 & $5$ & 24,025 & 6 & 1.8 & 4.38e-04 & 194,481 & 8 & 20 & 2.07e-06 \\
 & $6$ & 146,853 & 8 & 13 & 9.62e-05 & 1,322,578 & 8 & 148 & 1.73e-07 \\
\hline
\multirow{4}{*}{3} & $3$ & 833 & 9 & 0.8 & 5.40e-03 & 3825 & 8 & 1.1 & 3.57e-05 \\
 & $4$ & 3825 & 8 & 1.0 & 1.51e-03 & 24,025 & 9 & 4.0 & 1.98e-06 \\
 & $5$ & 24,025 & 9 & 3.2 & 3.52e-04 & 130,977 & 12 & 25 & 1.02e-07 \\
 & $6$ & 115,101 & 12 & 16 & 8.37e-05 & 790,321 & 13 & 167 & 1.57e-08 \\
\hline
\end{tabular}}
\end{footnotesize}
\end{center}
\end{table}

\begin{table}[!ht]
\caption{All-at-once solve of the heat equation: degrees of freedom (DoF), FGMRES iterations, CPU times, and resulting relative errors in $v$ for Lobatto IIIC methods, employing $\mathbf{Q}_1$ and $\mathbf{Q}_2$ finite elements.}\label{sequential_allatonce_heat_equation_table_Lobatto}
\begin{center}
\begin{footnotesize}
{\begin{tabular}{|c|c||c||c|c|c||c||c|c|c|}
\hline
\multicolumn{1}{|c}{} & \multicolumn{1}{|c||}{} & \multicolumn{4}{c||}{$\mathbf{Q}_1$} & \multicolumn{4}{c|}{$\mathbf{Q}_2$} \\
\cline{3-10}
$s$ & $l$ & DoF & $\texttt{it}$ & CPU & $v_{\text{error}}$ & DoF & $\texttt{it}$ & CPU & $v_{\text{error}}$ \\
\hline
\hline
\multirow{4}{*}{2} & $3$ & 1225 & 5 & 0.5 & 1.28e-02 & 11,025 & 6 & 2.1 & 2.20e-03 \\
 & $4$ & 11,025 & 6 & 1.7 & 3.56e-03 & 133,579 & 8 & 18 & 3.25e-04 \\
 & $5$ & 93,217 & 7 & 8.2 & 1.05e-03 & 1,528,065 & 8 & 163 & 4.56e-05 \\
 & $6$ & 766,017 & 8 & 65 & 2.71e-04 & 17,580,610 & 10 & 2512 & 5.81e-06 \\
\hline
\multirow{4}{*}{3} & $3$ & 833 & 6 & 0.5 & 5.42e-03 & 5625 & 7 & 1.3 & 9.82e-05 \\
 & $4$ & 5625 & 7 & 1.1 & 1.28e-03 & 39,401 & 7 & 5.1 & 1.58e-05 \\
 & $5$ & 31,713 & 7 & 3.3 & 3.63e-04 & 257,985 & 9 & 34 & 2.59e-06 \\
 & $6$ & 194,481 & 7 & 16 & 8.27e-05 & 1,758,061 & 10 & 292 & 3.68e-07 \\
\hline
\multirow{4}{*}{4} & $3$ & 1029 & 9 & 1.0 & 5.75e-03 & 4725 & 9 & 1.7 & 2.37e-05 \\
 & $4$ & 4725 & 9 & 1.4 & 1.54e-03 & 29,791 & 12 & 7.3 & 1.57e-06 \\
 & $5$ & 29,791 & 12 & 5.9 & 3.53e-04 & 162,729 & 12 & 31 & 1.38e-07 \\
 & $6$ & 142,884 & 12 & 22 & 8.40e-05 & 983,869 & 13 & 231 & 1.75e-07 \\
\hline
\multirow{4}{*}{5} & $3$ & 931 & 10 & 1.1 & 5.91e-03 & 5625 & 12 & 2.8 & 1.94e-05 \\
 & $4$ & 5625 & 12 & 2.4 & 1.54e-03 & 29,791 & 13 & 8.6 & 1.19e-06 \\
 & $5$ & 24,025 & 12 & 4.9 & 4.01e-04 & 146,853 & 13 & 33 & 2.40e-07 \\
 & $6$ & 123,039 & 12 & 20 & 9.64e-05 & 790,321 & 15 & 232 & 4.85e-08 \\
\hline
\end{tabular}}
\end{footnotesize}
\end{center}
\end{table}

\begin{table}[!ht]
\caption{All-at-once solve of the heat equation: degrees of freedom (DoF), FGMRES iterations, CPU times, and resulting relative errors in $v$ for Radau IIA methods, employing $\mathbf{Q}_1$ and $\mathbf{Q}_2$ finite elements.}\label{sequential_allatonce_heat_equation_table_Radau}
\begin{center}
\begin{footnotesize}
{\begin{tabular}{|c|c||c||c|c|c||c||c|c|c|}
\hline
\multicolumn{1}{|c}{} & \multicolumn{1}{|c||}{} & \multicolumn{4}{c||}{$\mathbf{Q}_1$} & \multicolumn{4}{c|}{$\mathbf{Q}_2$} \\
\cline{3-10}
$s$ & $l$ & DoF & $\texttt{it}$ & CPU & $v_{\text{error}}$ & DoF & $\texttt{it}$ & CPU & $v_{\text{error}}$ \\
\hline
\hline
\multirow{4}{*}{2} & $3$ & 931 & 6 & 0.5 & 5.09e-03 & 5625 & 6 & 1.0 & 2.02e-04 \\
 & $4$ & 5625 & 6 & 0.9 & 1.35e-03 & 47,089 & 7 & 5.3 & 2.51e-05 \\
 & $5$ & 38,440 & 6 & 3.1 & 3.46e-04 & 384,993 & 8 & 41 & 3.41e-06 \\
 & $6$ & 254,016 & 8 & 21 & 8.35e-05 & 3,112,897 & 8 & 351 & 5.71e-07 \\
\hline
\multirow{4}{*}{3} & $3$ & 833 & 7 & 0.5 & 5.71e-03 & 4725 & 8 & 1.2 & 2.57e-05 \\
 & $4$ & 4725 & 7 & 1.0 & 1.46e-03 & 27,869 & 9 & 4.4 & 2.28e-06 \\
 & $5$ & 27,869 & 9 & 3.7 & 3.30e-04 & 178,605 & 10 & 26 & 2.24e-07 \\
 & $6$ & 130,977 & 9 & 14 & 1.03e-04 & 1,048,385 & 12 & 204 & 2.68e-08 \\
\hline
\multirow{4}{*}{4} & $3$ & 784 & 9 & 0.7 & 5.59e-03 & 4725 & 10 & 1.7 & 1.94e-05 \\
 & $4$ & 4725 & 10 & 1.5 & 1.54e-03 & 24,986 & 10 & 4.8 & 1.12e-06 \\
 & $5$ & 24,986 & 10 & 4.0 & 3.77e-04 & 142,884 & 12 & 27 & 7.56e-08 \\
 & $6$ & 123,039 & 12 & 18 & 9.00e-05 & 741,934 & 15 & 196 & 6.63e-07 \\
\hline
\multirow{4}{*}{5} & $3$ & 931 & 12 & 1.2 & 5.91e-03 & 5625 & 14 & 3.2 & 1.92e-05 \\
 & $4$ & 5625 & 14 & 2.7 & 1.54e-03 & 24,025 & 13 & 6.4 & 1.16e-06 \\
 & $5$ & 24,025 & 14 & 5.7 & 4.01e-04 & 146,853 & 14 & 35 & 5.18e-08 \\
 & $6$ & 123,039 & 14 & 23 & 9.65e-05 & 693,547 & 13 & 171 & 1.32e-07 \\
\hline
\end{tabular}}
\end{footnotesize}
\end{center}
\end{table}

Tables \ref{sequential_allatonce_heat_equation_table_Gauss}--\ref{sequential_allatonce_heat_equation_table_Radau} show the parameter robustness of our preconditioner, with the linear solver converging in at most 15 iterations for all the methods and the parameters chosen. We are able to observe almost linear scalability of the solver with respect to the dimension of the problem. We would like to note that, when employing the finest grid for $\mathbf{Q}_2$ elements for the $2$-stage Lobatto IIIC method, the all-at-once discretization results in a system of 17,580,610 unknowns to be solved on a laptop. Finally, we also observe the predicted second- and third-order convergence for $\mathbf{Q}_1$ and $\mathbf{Q}_2$ discretization, respectively, until the linear tolerance causes slight pollution of the discretized solution for the finest grid of $\mathbf{Q}_2$ discretization (see, for instance, level $l=6$ for the 5-stage Radau IIA method in Table \ref{sequential_allatonce_heat_equation_table_Radau}).

\subsubsection{Stokes equations: sequential time-stepping}
In this section, we provide numerical results of the efficiency of the preconditioner $\widetilde{\mathcal{P}}_\mathrm{RK}$ defined in \eqref{prec_Stokes_equation} for the system of the stages for the Stokes problem. We run preconditioned GMRES up to a tolerance of $10^{-8}$ on the relative residual, with a restart after every 10 iterations.

As for the heat equation, we first present results for a sequential time-stepping. We take $t_f=2$, and solve the problem \eqref{Stokes_equation}, testing our solver against the following exact solution:
\begin{displaymath}
\begin{array}{l}
\vspace{0.2ex}
\vec{v}(\mathbf{x},t) = e^{t_f - t} [20x_1 x_2^3, 5 x_1^4 - 5x_2^4]^\top,\\
p(\mathbf{x},t) =e^{10t_f - t} \left( 60x_1^2x_2 - 20 x_2^3 \right)+\text{constant},
\end{array}
\end{displaymath}
with source function $\vec{f}$ as well as initial and boundary conditions obtained from this solution. The previous is an example of time-dependent \emph{colliding flow}. A time-independent version of this flow has been used in \cite[Section 3.1]{Norburn_Silvester}. We evaluate the errors in the $L^\infty(\mathcal{H}_0^1(\Omega)^d)$ norm for the velocity and in the $L^\infty(L^2(\Omega))$ norm for the pressure, defined respectively as follows:
\begin{displaymath}
\begin{array}{c}
\displaystyle v_{\text{error}} = \max_{n} \left[ (\mathbf{v}_n - \mathbf{v}_n^{\mathrm{sol}})^\top K_v (\mathbf{v}_n - \mathbf{v}_n^{\mathrm{sol}}) \right]^{1/2}, \\
\displaystyle p_{\text{error}} = \max_{n} \left[ (\mathbf{p}_n - \mathbf{p}_n^{\mathrm{sol}})^\top M_p (\mathbf{p}_n - \mathbf{p}_n^{\mathrm{sol}}) \right]^{1/2}.
\end{array}
\end{displaymath}
In the previous, $\mathbf{v}_n^{\mathrm{sol}}$ (resp., $\mathbf{p}_n^{\mathrm{sol}}$) is the discretized exact solution for $\vec{v}$ (resp., $p$) at time $t_n$.

In Tables \ref{sequential_Stokes_equations_table_Gauss}--\ref{sequential_Stokes_equations_table_Lobatto_Radau}, we report the average number of GMRES iterations and the average CPU times together with the numerical errors on the solutions, for different Runge--Kutta methods. Further, in Table \ref{sequential_Stokes_equations_table_Lobatto_Radau} we report the dimensions of the system solved for Lobatto IIIC and Radau IIA methods. As for the heat equation, those values coincide with the dimensions of the system arising when employing the corresponding level of refinement in an $s$-stage Gauss method.

\begin{table}[!ht]
\caption{Sequential solve of the Stokes equations: average GMRES iterations, average CPU times, and resulting errors in $\vec{v}$ and $p$, for Gauss methods.}\label{sequential_Stokes_equations_table_Gauss}
\begin{center}
\begin{footnotesize}
{\begin{tabular}{|c||c|c|c|c||c|c|c|c|}
\hline
\multicolumn{1}{|c||}{} & \multicolumn{4}{c||}{$s=2$} & \multicolumn{4}{c|}{$s=3$} \\
\cline{2-9}
\cline{2-9}
$l$ & $\texttt{it}$ & CPU & $v_{\text{error}}$ & $p_{\text{error}}$ & $\texttt{it}$ & CPU & $v_{\text{error}}$ & $p_{\text{error}}$ \\
\hline
\hline
$3$ & 36 & 0.15 & 1.29e+00 & 9.73e-01 & 54 & 0.36 & 1.42e+00 & 7.19e-01 \\
$4$ & 37 & 0.34 & 1.66e-01 & 4.86e-01 & 61 & 0.97 & 1.85e-01 & 1.20e-01 \\
$5$ & 42 & 1.5 & 2.08e-02 & 2.15e-01 & 68 & 3.3 & 2.39e-02 & 3.20e-02 \\
$6$ & 47 & 6.8 & 2.57e-03 & 1.20e-01 & 75 & 17 & 3.07e-03 & 1.10e-02 \\
$7$ & 50 & 32 & 3.12e-04 & 5.94e-02 & 83 & 85 & 3.96e-04 & 3.87e-03 \\
\hline
\end{tabular}}
\end{footnotesize}
\end{center}
\end{table}

\begin{table}[!ht]
\caption{Sequential solve of the Stokes equations: degrees of freedom (DoF), average GMRES iterations, average CPU times, and resulting errors in $\vec{v}$ and $p$, for Lobatto IIIC and Radau IIA methods.}\label{sequential_Stokes_equations_table_Lobatto_Radau}
\begin{center}
\begin{footnotesize}
{\begin{tabular}{|c|c||c||c|c|c|c||c|c|c|c|}
\hline
\multicolumn{1}{|c}{} & \multicolumn{1}{|c||}{} & \multicolumn{1}{c||}{} & \multicolumn{4}{c||}{Lobatto IIIC} & \multicolumn{4}{c|}{Radau IIA} \\
\cline{4-11}
$s$ & $l$ & DoF & $\texttt{it}$ & CPU & $v_{\text{error}}$ & $p_{\text{error}}$ & $\texttt{it}$ & CPU & $v_{\text{error}}$ & $p_{\text{error}}$ \\
\hline
\hline
\multirow{4}{*}{2} & $3$ & 1062 & 33 & 0.14 & 1.41e+00 & 5.67e-01 & 35 & 0.15 & 1.37e+00 & 5.07e-01 \\
 & $4$ & 4422 & 39 & 0.36 & 2.04e-01 & 5.19e-02 & 38 & 0.37 & 1.77e-01 & 4.20e-02 \\
 & $5$ & 18,054 & 46 & 1.5 & 3.19e-02 & 5.11e-03  & 44 & 1.4 & 2.27e-02 & 3.53e-03 \\
 & $6$ & 72,966 & 50 & 6.8 & 5.25e-03 & 5.45e-04 & 49 & 6.9 & 2.91e-03 & 3.13e-04 \\
 & $7$ & 293,382 & 54 & 34 & 8.95e-04 & 6.20e-05 & 54 & 36 & 3.90e-04 & 3.36e-05\\
\hline
\multirow{4}{*}{3} & $3$ & 1593 & 42 & 0.26 & 1.36e+00 & 5.06e-01 & 48 & 0.28 & 1.35e+00 & 5.01e-01 \\
 & $4$ & 6633 & 44 & 0.64 & 1.76e-01 & 4.18e-02 & 55 & 0.79 & 1.75e-01 & 4.16e-02 \\
 & $5$ & 27,081 & 54 & 2.7 & 2.24e-02 & 3.54e-03 & 64 & 3.1 & 2.23e-02 & 3.53e-03 \\
 & $6$ & 109,449 & 62 & 14 & 2.82e-03 & 3.04e-04 & 74 & 17 & 2.82e-03 & 3.02e-04 \\
 & $7$ & 440,073 & 70 & 73 & 3.55e-04 & 2.64e-05 & 78 & 85 & 3.54e-04 & 2.61e-05 \\
\hline
\multirow{4}{*}{4} & $3$ & 2124 & 56 & 0.46 & 1.35e+00 & 4.99e-01 & 67 & 0.54 & 1.36e+00 & 4.97e-01 \\
 & $4$ & 8844 & 67 & 1.3 & 1.75e-01 & 4.14e-02 & 69 & 1.4 & 1.75e-01 & 4.07e-02 \\
 & $5$ & 36,108 & 75 & 5.2 & 2.23e-02 & 3.49e-03 & 79 & 5.4 & 2.22e-02 & 3.45e-03 \\
 & $6$ & 145,932 & 80 & 25 & 2.81e-03 & 2.99e-04 & 87 & 28 & 2.80e-03 & 2.96e-04 \\
 & $7$ & 586,764 & 88 & 131 & 3.53e-04 & 2.59e-05 & 91 & 135 & 3.52e-04 & 2.56e-05 \\
\hline
\multirow{4}{*}{5} & $3$ & 2655 & 68 & 0.74 & 1.34e+00 & 4.88e-01 & 76 & 0.83 & 1.34e+00 & 4.89e-01 \\
 & $4$ & 11,055 & 82 & 2.2 & 1.75e-01 & 4.08e-02 & 90 & 2.3 & 1.74e-01 & 4.10e-02 \\
 & $5$ & 45,135 & 96 & 8.9 & 2.22e-02 & 3.46e-03 & 93 & 8.3 & 2.21e-02 & 3.43e-03 \\
 & $6$ & 182,415 & 99 & 43 & 2.80e-03 & 2.94e-04 & 105 & 45 & 2.80e-03 & 2.95e-04 \\
 & $7$ & 733,455 & 110 & 242 & 3.52e-04 & 2.55e-05 & 110 & 247 & 3.52e-04 & 2.55e-05 \\
\hline
\end{tabular}}
\end{footnotesize}
\end{center}
\end{table}

From Tables \ref{sequential_Stokes_equations_table_Gauss}--\ref{sequential_Stokes_equations_table_Lobatto_Radau}, we can observe that the mesh size $h$ is slightly affecting the average number of GMRES iterations\footnote{From the authors' experience, this is because we are approximating the $(1,1)$-block in conjunction with the block-commutator argument. In fact, the block-commutator preconditioner would have been more effective if one employed a fixed number of GMRES iterations to approximately invert the $(1,1)$-block. However, having a further ``layer'' of preconditioner would have been non-optimal in the all-at-once approach.}. Further, we can observe a more invasive dependence on the number of stages $s$. Nonetheless, the CPU times scale almost linearly with respect to the dimension of the system solved. Finally, we can observe that the error on the velocity and on the pressure behaves (roughly) as second order for every method but the $2$-stage Gauss method, for which we observe a first-order convergence for the error on the pressure.

\subsection{Parallel results}
We can now show the parallel efficiency of our solver. The parallel code is written in Fortran 95 with pure MPI as the message passing protocol, and compiled
with the \texttt{-O5} option.
We run our code on the M100 supercomputer, an IBM Power AC922, 
with 980 computing nodes and 2$\times$16 cores at 2.6(3.1) GHz on each node.
The local node RAM is about 242 GB.

Throughout the whole section we will denote with $T_\p$ the CPU elapsed times
expressed in seconds when running the code on $\p$ processors.
We include a measure of the parallel
efficiency achieved by the code. To this aim we will  denote as
$S_\p$ and $E_\p$ the speedup and parallel efficiency, respectively, defined as
\[ S_\p = \frac{T_1}{T_\p}, \qquad 
E_\p= \dfrac{S_\p}{\p}  = \frac{T_{1}}{T_\p {\p}} . \]
Moreover $T_\p^{\Theta}$ and  $T_\p^{S}$ refer to the cost of   approximate inversion of the two diagonal
blocks of the preconditioner, which are the most costly operation in the whole algorithm.

\noindent
\begin{remark}
Since we could not run the largest problems with $\p = 1$ for obvious memory and CPU time reasons, we 
approximate $T_1$ from below under the following assumptions:
\begin{enumerate}
	\item The approximate inversion of the $(2,2)$-block $\Theta$ is completely scalable as it does not require any communication among processors, hence for any $\p$, $T_1^{\Theta} = \p T_{\p}^{\Theta}$.
	\item Since in sequential computations $T_1^{\Theta}  < T_1^{S}$,  we set  $T_1^{S}  = T_1^{\Theta} $.
	\item We therefore use $T_1 = 2 T_1^{\Theta}  =  2 \p T_{\p}^{\Theta}$.
\end{enumerate}
With these assumptions, as we are underestimating $T_1$, we also underestimate both $S_\p$ and $E_\p$ for our code.
\end{remark}

We adopt a uniform distribution of the data (cofficient matrix and right hand side) among processors.
We assume that the number of processors $\p$ divides the number of time-steps $n_t$, so that each processor
manages $n_t/\p$ time-steps, and the whole space-discretization matrices are replicated on each processor.

\subsubsection{XBraid}
To solve the system with $\widehat S$ at  every FGMRES iteration we employed the Parallel Multigrid in time software: XBraid \cite{Falgout_Friedhoff_Kolev_MacLachlan_Schroder}.
After some preliminary testing on smaller problems, we set the relevant parameters as follows:
\smallskip

\begin{itemize}
\item
Coarsening factor: \texttt{cfactor} $=4$, \\[-1.0em]
\item
	 Minimum coarse grid size: \texttt{min\_coarse} $=2$,  \\[-1.0em]
\item
Maximum number of iterations: \texttt{max\_its} $=1$. 
\end{itemize}
We refer the reader to the manual \cite{XBraid} for information on the meaning of these parameters, and for a description of the XBraid package.

\subsubsection{Description of test cases}
We solve two variants of the heat equation problem described in Section \ref{heat_equation:sequential_time_stepping}, for two different final time $t_f$ (Test cases \#1 and \#2). Further, for the Stokes equations we consider a time-dependent version of the \emph{lid-driven cavity} \cite[Example\,3.1.3]{Elman_Silvester_Wathen}. Starting from an initial condition $\vec{v}(\mathbf{x},0)=\vec{0}$, the flow is described by \eqref{Stokes_equation} with $\vec{f}(\mathbf{x}, t)=\vec{0}$ and the following boundary conditions:
\begin{displaymath}
\vec{g}(\mathbf{x},t)=
\left\{
\begin{array}{cl}
\left[t,0\right]^\top & \mathrm{on} \: \partial \Omega_1 \times (0,1),\\
\left[1,0\right]^\top & \mathrm{on} \: \partial \Omega_1 \times [1,t_f),\\
\left[0,0\right]^\top & \mathrm{on} \: (\partial \Omega \setminus \partial\Omega_1) \times (0,t_f).
\end{array}
\right.
\end{displaymath}
Here, we set $\partial \Omega_1 := \left(-1,1 \right)\times \left\{1\right\}$. Within this setting, the system is able to reach a steady state, see \cite{Shankar_Deshpande}. We would like to note that, in order for a system of this type to reach stationarity, one has to integrate the problem up to about 100 time units. However, in our tests we integrate the problem for longer times. Specifically, we solve the problem up to $t_f = 386.8$ (Test case \#3) and $t_f = 510.7$ (Test case \#4). This is done as a proof of concept of how Runge--Kutta methods in time can be applied within a parallel solver in order to integrate problems for very long times.

In Table \ref{parallel_settings}, we report the settings for the different tests together with the total number of degrees of freedom DoF and the number of FGMRES iterations $\texttt{it}$.
\noindent
\begin{table}[h!]
\begin{center}
	\caption{Description of the test cases for parallel runs. The last column reports the outer FGMRES iterations.}
	\label{parallel_settings}
	\renewcommand{\arraystretch}{1.2}
	\begin{footnotesize}
	\begin{tabular}{|c|l|l|c||c|c||r|r|c||c||c|}
	\hline
		& & & & \multicolumn{2}{c||}{order} & && & &\\ \cline{5-6}
		\#& Eq. & RK method & $t_f$ & {\small{FE}}\ &  {\small RK}& $n_t$  &$l$&  $s$ & DoF& $\texttt{it}$ \\
		\hline \hline
		1 &Heat &	Radau IIA &$  255.9 $& 3 & 7 &  $2048$&8& $4$ & $2,674,140,160$ & 18\\
		2 &Heat &	Lobatto IIIC &$ 2047.9$& 3 & 6 & $16,384$&7&  $4$ &$5,326,913,024$ & 13 \\
		3 &Stokes &	Radau IIA &$  386.8 $& 2 & 5 &  $2048$&7& $3$ & $1,201,839,360 $& 23 \\
		4 &Stokes &	Lobatto IIIC &$  510.7 $& 2& 6 &  $2048$&7& $4$ & $1,502,262,528 $& 28 \\
		\hline
	\end{tabular}
	\end{footnotesize}
\end{center}
\end{table}

In the next subsections we present the results of a strong scalability study for these four test cases.

\subsubsection{Heat equation}
The parallel results corresponding to Test cases \#1 and \#2 are summarized in Tables \ref{t1} and Figure \ref{f1}.
Extremely satisfactory parallel efficiency above 75\% with $\p \le 128$ (Test case \#1) and  $\p \le 512$ (Test case \#2) is achieved.
This is also accounted for by Figure \ref{f1} (left) which compares the ideal CPU times (corresponding to 100\% efficiency, red line) with the 
measured CPU times (blue line). Parallel efficiency degrades to 30$\%$ only when the ratio of the number of time-steps and the number of processors is smaller than $8$ (see, for instance, $\p=4096$).
\begin{table}[h!]
	\caption{Parallel solve of the heat equation: CPU time, runtime speedup and parallel efficiencies for 
	Test case \#1. }
	\label{t1}
\begin{center}
\renewcommand{\arraystretch}{1.15}
\begin{footnotesize}
	\begin{tabular}{|c||c|c|c||c|c|}
	\hline
		$\p$ & $T_\p^{\Theta}$ & $T_\p^{S}$  &$T_\p$  & $S_\p$ & $E_\p$  \\
	\hline \hline
		64 &   23,845  & 28,607  &52,549 & 58  & 91\% \\
		128 & 11,925  &19,133 &31,108 &98 & 77\% \\
		256 &   5991 & 15,523  &21,542 & 143   & 56\%   \\
		512 &   3012 & 13,761  &16,790 & 184   & 36\%   \\
		\hline
\end{tabular}
\end{footnotesize}
\end{center}
\end{table}

\begin{figure}[h!]
	\caption{Parallel solve of the heat equation: CPU time, runtime speedup and parallel efficiencies for 
	Test case \#2. On the right, CPU times vs ideal ones, computed assuming 100\% efficiency.}
	\label{f1}
\begin{minipage}{6.8cm}
\begin{center}
\renewcommand{\arraystretch}{1.15}
\begin{footnotesize}
	\begin{tabular}{|c||c|c|c||c|c|}
	\hline
		$\p$ & $T_\p^{\Theta}$ & $T_\p^{S}$  &$T_\p$  & $S_\p$ & $E_\p$ \\
	\hline \hline
		256 &  4924   & 5920  &10,879 & 232  & 91\%\\
		512  & 2519   & 3709  &6249 & 414  & 81\%\\
		 1024  & 1262 &  2570  & 3843 & 674   & 66\%  \\
		 2048 & 637 &  2138  & 2783 & 941   & 46\% 	\\	 
		 4096 & 308 &  1778  & 2286 & 1210   & 30\% 	\\	 
		 \hline
\end{tabular}
\end{footnotesize}
\end{center}
	\hspace{-6mm}
\end{minipage}
	\begin{minipage}{6.6cm}
\begin{center}
\includegraphics[width=6.0cm]{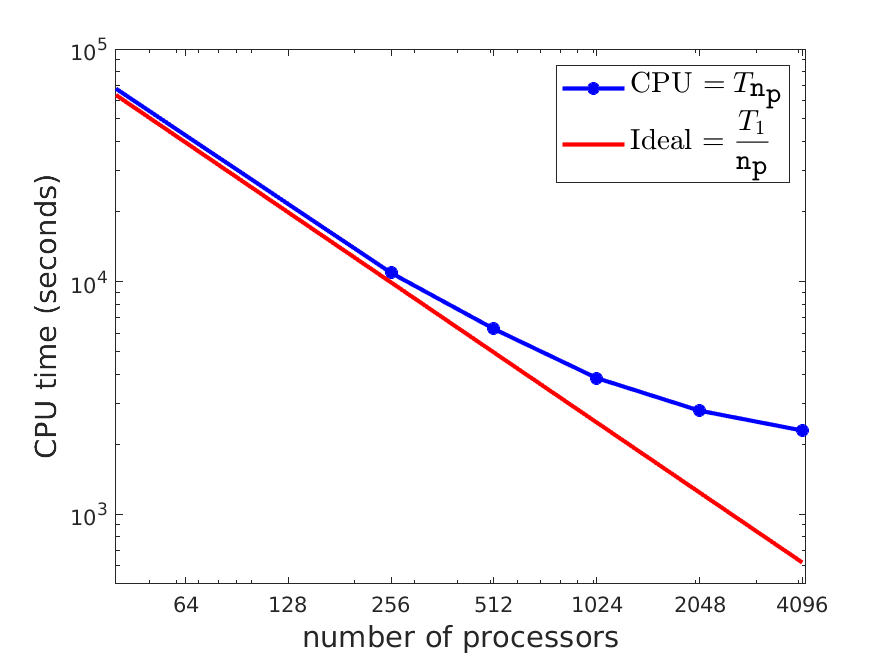}
\end{center}
\end{minipage}
\end{figure}

\subsubsection{Stokes equations}
We report here the parallel results for the simulation of the lid-driven cavity problem. For this test case, each block
of the preconditoner is approximately inverted with 10 GMRES (inner) iterations, while FGMRES is restarted
after 30 iterations, and stops whenever the relative residual is below the tolerance of $10^{-6}$. The results confirm that significant runtime speedups are achieved by the proposed approach.

\begin{table}[h!]
\begin{center}
	\caption{Parallel solve of the Stokes equations: CPUs, speedups, and parallel efficiencies for 
	Test cases \#3 (left) and \#4 (right). }
	\label{t3}
	\renewcommand{\arraystretch}{1.15}
	\begin{footnotesize}
	\begin{tabular}{|c||c|c|c|c|c||c|c|c|c|c|}
	\hline
		   & \multicolumn{5}{c||}{Radau IIA} &
		    \multicolumn{5}{c|}{Lobatto IIIC}  \\
		    \cline{2-11}
		$\p$ & $T_\p^{\Theta}$ & $T_\p^{S}$  &$T_\p$  & $S_\p$ & $E_\p$
		&$T_\p^{\Theta}$ & $T_\p^{S}$  &$T_\p$  & $S_\p$ & $E_\p$\\
	\hline \hline
		64 &   19,165  & 25,174  &44,410 & 55  & 86\% &31,865&39,939&71,920& 57& 89\%  \\
		128 & 9628  & 16,471 &26,140 &94 & 74\% &15,903&26,167&42,130&97 & 76\%  \\
		256 &   4830 & 13,197  &18,048 & 137   & 54\% &
		   7726 & 20,149  &28,883 & 142   & 55\%  \\
		   512 &2876&11,603&14,510&203&40\%& 4026 & 17,070 & 21,859 & 195 & 38\% \\
		   \hline
\end{tabular}
\end{footnotesize}
\end{center}
\end{table}

\section{Conclusions}\label{sec_5}
In this work, we presented a robust preconditioner for the numerical solution of the all-at-once linear system arising when employing a Runge--Kutta method in time. The proposed preconditioner consists of a block diagonal solve for the systems for the stages of the method, and a block-forward substitution for the Schur complement. Since we require a solver for the system for the stages, we proposed a preconditioner for the latter system based on the SVD of the Runge--Kutta coefficient matrix $A_{\mathrm{RK}}$. Sequential and parallel results showed the efficiency and robustness of the proposed preconditioner with respect to the discretization parameters and the number of stages of the Runge--Kutta method employed for a number of test problems.

Future work includes adapting this solver for the numerical solution of more complicated time-dependent PDEs as well as employing an adaptive time-stepping technique. Further, we plan to adopt a hybrid parallel programming paradigm. 
In this way, each of the $n_t/\p$ time-steps sequentially processed by each MPI rank in the present approach could be handled in parallel by different Open-MP tasks. Finally, for extremely fine space discretizations, parallelization in space will also be exploited.

\section*{Acknowledgements}
LB and AM gratefully acknowledge the support of the ``INdAM -- GNCS Project'', CUP\_E53C22001930001. JWP gratefully acknowledges financial support from the Engineering and Physical Sciences Research Council (EPSRC) UK grant EP/S027785/1. We acknowledge the CINECA award under the ISCRA initiative, for the availability of high performance computing resources and technical support. SL thanks Michele Benzi for useful discussions on Theorem \ref{optimality_preconditioner} and on the preconditioner for the Stokes equations. We thank Robert Falgout for useful discussions on the XBraid software.
	The idea of this work arose while AM was visiting the School of Mathematics of the University of Edinburgh under the Project HPC-EUROPA3 (INFRAIA-2016-1-730897), with the support of the EC Research Innovation Action under the H2020 Programme.

%
%
%

\end{document}